\documentclass[11pt]{article}      
\usepackage{amsmath,amsthm,amsfonts,graphicx}
\usepackage[english]{babel}
\usepackage{tikz}
\usepackage{graphicx}
\usepackage{xcolor}
\usepackage{mathtools} 
\usepackage{graphics}
\usepackage{subcaption}
\usepackage{enumitem}
\usepackage{hyperref}
\DeclarePairedDelimiter{\ceil}{\lceil}{\rceil} %ceiling function 
 
\def\smallddots{\mathinner{\raise7pt\hbox{.}\raise4pt\hbox{.}\raise1pt\hbox{.}}} 
\def\smallsdots{\mathinner{\raise1pt\hbox{.}\raise4pt\hbox{.}\raise7pt\hbox{.}}} 

\DeclareMathOperator*{\argmin}{arg\,min}   % declare argmin
\DeclareMathOperator*{\dist}{Dist}

\DeclareMathOperator{\diag}{diag}

\DeclareMathOperator{\rank}{rank}

\DeclareMathOperator{\range}{range}

\newtheorem{theorem}{Theorem}[section]

\numberwithin{equation}{section}
\numberwithin{table}{section}
\newtheorem{lemma}{Lemma}[section] 

\newtheorem{corollary}{Corollary}[section]

\newtheorem{algorithm}{Algorithm}[section]
\newtheorem{example}{Example}[section]
\newtheorem{definition}{Definition}[section]

\newtheorem{remark}{Remark}[section]

\begin{document} 
  
\title{Superfast Iterative Reﬁnement of\\ Low Rank Approximation of a Matrix\\ Based on the ALS Method 
and Random Sampling\footnote{The results of this paper have been presented at minisymposia of ILAS  2023 and ICIAM 2023.}} 
\author{Victor Y. Pan} 
\author{Qi Luan$^{[1],[a]}$ and  
Victor Y. Pan$^{[1,2],[b]}$\\ %VP??and John Svadlenka$^{[2],[c]}$\\
%, and  Liang Zhao$^{[1, 2],[d]}$
$^{[1]}$ Ph.D. Programs in  Computer Science and Mathematics \\
The Graduate Center of the City University of New York \\
New York, NY 10036 USA \\ 
%\and\\
$^{[2]}$ Department of Computer Science \\
Lehman College of the City University of New York \\
Bronx, NY 10468 USA \\ 
$^{[a]}$ qi\_luan@yahoo.com \\  
$^{[b]}$ victor.pan@lehman.cuny.edu \\ 
http://comet.lehman.cuny.edu/vpan/  \\
%$^{[c]}$ 
%VP??jsvadlenka@gradcenter.cuny.edu \\ %$^{[d]}$ Liang.Zhao1@lehman.cuny.edu \\
}
 
\date{}  

\maketitle 

%------------------------------------------------------------------------------
%------------------------------------------------------------------------------

\begin{abstract}
A matrix algorithm runs superfast (aka  at {\em sublinear cost}) if it involves much fewer flops and memory cells than an input matrix has entries.
 Big Data are frequently   represented by matrices of immense sizes that cannot be handled 
directly but can be  approximated with low rank matrices, with which one can operate superfast.  Superfast
computation of {\em Low Rank Approximations (LRA)} of a matrix, however, can be a challenge. Any superfast  LRA algorithm fails miserably  on worst case matrices. Fortunately, they rarely appear in computational practice. For an important example,  superfast  Adaptive Cross--Approximation iterations have consistently output accurate LRA of a large and important class of matrices
during the decades of their worldwide  application.
Furthermore, they output CUR LRA, which is a special attractive form of LRA, widely applied in data analysis.
%VP Adequate formal support for that empirical behavior, however, is still a challenge.
Encouraged by success of these iterations  
we present a 
superfast  randomized iterative refinement of a crude CUR LRA by 
means of combining   random sampling  with the Alternating Least Squares Method, which  reduces this refinement to recursive solution of  generalized Linear Least Squares Problems (LLSPs). 
 We prove  monotone convergence of our iterations with a high probability to a near-optimal LRA of  an input matrix   
under  some specified assumptions on it and on  a crude  approximate solution to an initial generalized LLSP.
In our numerical tests two or three iterations of our algorithm have consistently and significantly improved crude initial LRAs of
 real world inputs.
\end{abstract}


%\bigskip
%\medskip 
\clearpage

{\bf Abbreviations and Explanations/Pointers}
\smallskip

ACA iterations, Adaptive Cross-Approximation iterations (Secs. 1.1 and \ref{srltd})

ALS method,   Alternating Least Squares method
\cite{CHY12,CZPA09,
JNS13,
KBV09,ORU22}

CUR (Sec. \ref{scur}) 

LLSP, Linear Least-squares Problem  (Eqn. (\ref{eqllsps}))

LRA, Low rank approximation (Sec. \ref{slra}) 

%\end{table}

RRQR factorization,  Rank Revealing QR  factorization \cite{GE96}
\medskip

{\bf  Definitions and Explanations/Pointers}

$||\cdot||$, the spectral norm (Secs. 1.1 and \ref{slap})

$||\cdot||_F$ , the Frobenius norm (Secs. 1.1 and \ref{slap})

$\beta$, a positive parameter not exceeding 1 (Remark \ref{rescrtoprb})

$\delta:=\dist(A,U^{(r)}) \le 1/2$,
 $\theta:=
\frac{\Bar\sigma_{r+1}}{\sigma_{r+1}}$ (Cor. \ref{coro:iter_ref})

 $\gamma_i$, $\Tilde \gamma_i$, leverage scores 
(Definition \ref{def:leverage_score})  

$\epsilon$, $\xi$, and $e$, other bounded positive parameters in  Algs. \ref{algsmplex} and \ref{algalter}  
 and Thm. \ref{thm:contracting_distance}

 $\sigma_j(M)$, the $j$th largest singular value of matrix $M$
 (Sec. 1.1 and Def. \ref{defsvd})

 $\sigma_{F,r+1}(M):=||M_r-M||_F$, (Sec. \ref{slra})
 
$A_t$ and $A$, some matrices in $\mathbb R^{m\times r}$ (Eqn. (\ref{eqllsps}))

$B_t$ and $B$,  some matrices in $\mathbb R^{r\times n}$ (Eqn. (\ref{eqllsps})

$c<1$,  a positive parameter, not exceeding 1 such that $\frac{1}{c}$ is the rate of linear convergence
of iterative refinement  of LRA (Eqns. (\ref{eqc}) and (\ref{eqn:con_dist2}) and Thm. \ref{algsmplex})

$\widehat D$, Scaling  matrix 
(Sec. \ref{ssrcs})

$\dist(G,H)$, the Principal Angle Distance
between  matrices $G$ and $H$ (Def. \ref{defpad})

$l$,  the number of row or column samples
(Sec. \ref{ssrcs})

$M$, a matrix in $\mathbb R^{m\times n}$

 $M^+$, the Moore-Penrose pseudo inverse of $M$ (Sec. \ref{slra})
 
 $M_{\perp}$ for a matrix $M$ satisfies 
 $\range(M_{\perp})=
\range(M)_{\perp}$ (Sec. \ref{slap}) 

$M_r$, the $r$-truncation of a matrix $M$, being an optimal rank-$r$ approximation of $M$ under both spectral and Frobenius matrix norms (Sec. 1.1 and Def. \ref{deftpsvd})

$M^T$, the transpose of a matrix $M$

$p_i$ and $\tilde p_i$, Sampling Probabilities
(Remark \ref{rescrtoprb}) 
 
$r$, target rank for LRA  
 
 $\range(M)$,  the range of a matrix $M$ (Sec. \ref{slap})

$\mathbb R^{g\times h}$,  the space of $g\times h$ real matrices
 
$r$-top SVD of $M$ is the compact SVD  $M_{r}=U^{(r)}\Sigma^{(r)} V^{(r)T}$ (Sec. 1.1 and Def. \ref{deftpsvd})

$S$, Sampling and Scaling  matrix 
(Sec. \ref{ssrcs})

$\widehat S$,  Sampling matrix 
(Sec. \ref{ssrcs})

$\mathbb S_{\perp}$ for a  subspace $\mathbb S$, its orthogonal complement
(Sec. \ref{slap})

SVD (Singular Value Decomposition) of $M$, the matrix equation $M=U_M\Sigma_MV^T_M$
(Def. \ref{defsvd})

%------------------------------------------------------------------------------
 \section{Introduction}
 %\bigskip
 
%------------------------------------------------------------------------------

\subsection{Superfast  LRA: motivation and the known algorithms}
  %\footnote{In the study of LRA it is customary to rely on an informal basic concept of low rank; we also use other informal concepts such as ``large",  ``small", ``ill-" and ``well-conditioned",  ``near", ``close", and ``approximate", 
%quantified in context,  in  high level introductory part of our  presentation  complemented later with formal  analysis.}
{\em Low Rank Approximation (LRA)} of a matrix is among the most fundamental subjects of Numerical Linear Algebra and Data Mining and Analysis, 
with applications ranging from PDEs to machine learning theory, neural networks, term document data, and DNA SNP data (see \cite{MD09,M11,HMT11,MT20}, 
and the bibliography therein). 
Quite typically, $m\times n$ matrices $M$ involved in Big Data applications (e.g., unfolding matrices of multidimensional tensors) are so immense that one can only handle them by applying superfast  (aka  sublinear cost) algorithms,\footnote{Superfast algorithms have been studied extensively for Toeplitz, Hankel, Vandermonde, Cauchy, and other structured matrices having small displacement rank \cite{P15}, and we extend this concept to other matrices.} which involve 
much fewer than  $mn$ memory cells and 
flops,\footnote{``Flop" stands for
``floating point arithmetic operation".} 
e.g., operate with LRA of a matrix $M$ whenever such an LRA is available.

Quite typically,  computations with Big Data are represented with matrices that admit LRA \cite{UT19}, but any algorithm  for computing LRA fails miserably on a worst case input matrices unless it involves all  their entries (see
 Example  \ref{exdlt}).
Superfast LRA is not just a dream, however.
The celebrated {\em Adaptive Cross-Approximation (ACA) iterations}, which adapt to LRA the Alternating Directions  Implicit (ADI) method
\cite{LW91,
S16},  consistently compute accurate LRAs   in their worldwide computations for a large class of  real world   matrices.  
Adequate formal support for this empirical behavior is still a challenge (see Sec. \ref{srltd}), but  
it is proved that superfast algorithms
of \cite{MW17,BW18,
LP20,CETWa}
compute near-optimal
LRA of a Symmetric Positive
Semidefinite (SPSD) matrix and a Distance matrix.\footnote{The  LRA algorithms of \cite{MW17,BW18,
CETWa} are randomized, and one cannot estimate   superfast their  output error norms to verify their correctness, but the superfast deterministic  algorithm of \cite[Part III]{LP20} computes CUR LRA of an $n\times n$ SSPD matrix having both spectral and Frobenius  error norms within a factor of $n$ from optimal and  as by-product estimates error
  norm,   verifying correctness -- at no additional cost.}

\subsection{Some background (briefly)}

The $r$-{\em truncation} $M_r$ of $M$ is  obtained by means of  setting to 0 all singular values of $M$ but  the
$r$ top  ones,
$\sigma_1(M),\dots,\sigma_r(M)$. 
[Here and hereafter $\sigma_j(M)$ denotes the $j$th largest singular value of $M$.] $r$-truncation  is unique if $\sigma_r(M)>
\sigma_{r+1}(M)$.
By virtue of  Eckart-Young-Mirsky's theorem $M_r$ is an optimal rank-$r$ approximation\footnote{Here and hereafter ``rank-$r$" means
``rank at most $r$". } of an $m\times n$ matrix $M$ for any positive $r\le \min\{m,n\}$, under both spectral and Frobenius matrix norms, denoted  $||\cdot||$ and 
 $||\cdot||_F$, respectively. 

Compact SVD of $M_r$ is said to be the $r$-{\em top SVD of} $M$. Its deterministic  computation   
 is quite slow for a matrix $M$ of  a large size \cite[Fig. 8.6.1]{GL13},
 but  the random sampling and sketching  algorithms of
\cite{DMM08,HMT11,
MT20,
TYUC17,N20,TWa} approximate it closely  with a high probability {\em (whp)} by running
 much faster, and even superfast except for  the stage of sampling or  sketching.
 
%  $FM$ and/or $MH$ for dense random matrices $F$ and $H$. 

\subsection{Iterative refinement of LRA}

Our goal is {\em iterative refinement} of  a crude LRA  by means of acceleration of the {\em Alternating  Least Squares (ALS)} method, 
 which specializes {\em Alternating Optimization}
(aka {\em Alternating Minimization)} method
(see \cite{CHY12,CZPA09,
JNS13,
KBV09,ORU22}, and the bibliography therein).

 Namely, given two  matrices $M\in \mathbb R^{m\times n}$ and $A_0\in \mathbb R^{m\times r}$ where $r\le \min\{m,n\}$ denotes target rank\footnote{``In practice, the target rank is rarely known in advance. $\dots$ LRA algorithms are
usually implemented in an adaptive fashion $\dots$
until the error norm satisfies the desired tolerance"
\cite[Sec. 2.4]{TYUC17}.} and where $A_0$ represents
 an initial approximation to the matrix of the left singular vectors of $M_r$, 
  one recursively
  solves the associated generalized Linear Least Squares Problems  {\em (LLSPs)} by computing matrices $B_{t+1}\in \mathbb R^{r\times n}$ and $A_{t+1}\in \mathbb R^{m\times r}$, for $t=0,1,\dots,\tau-1$ and a fixed integer $\tau>0$   
  (see Fig. \ref{fig:alt-ref-alg})  such that

\begin{figure}[hbt!]
\begin{subfigure}{.45\linewidth}
    \centering
    \resizebox{!}{3cm}{
        \begin{tikzpicture}
            \filldraw[color=black, fill=black!5, very thick] (0, 0) rectangle (5, 5);
            \filldraw[color=black, fill=black!5, very thick] (6, 0) rectangle (7, 5);
            \filldraw[color=black, fill=black!5, very thick] (7.5, 5) ;
            \draw (2.5, 2.5) circle (0pt) node{$M$};
            \draw (6.5, 2.5) circle (0pt) node{$A_0$};
            \draw (0, -0.5) circle (0pt) node[anchor=west]{Matrices $M$ and $A_0$ are an initial input};
        \end{tikzpicture}
    }
    \caption{}
\end{subfigure}\
\begin{subfigure}{.45\linewidth}
    \centering
    \resizebox{!}{3cm}{
        \begin{tikzpicture}
            \filldraw[color=black, fill=black!5, very thick] (0, 0) rectangle (5, 5);
            \filldraw[color=black, fill=black!30, very thick] (6, 0) rectangle (7, 5);
            \filldraw[color=black, fill=green!15, very thick] (7.5, 5) rectangle (12.5, 4);
            \draw (2.5, 2.5) circle (0pt) node{$M$};
            \draw (6.5, 2.5) circle (0pt) node{$A_0$};
            \draw (10, 4.5) circle (0pt) node{$B_1$};
            \draw (0, -0.5) circle (0pt) node[anchor=west]{Compute $B_1$};
        \end{tikzpicture}
    }
    \caption{}
\end{subfigure}

\medskip

\begin{subfigure}{.45\linewidth}
    \centering
    \resizebox{!}{3cm}{
        \begin{tikzpicture}
            \filldraw[color=black, fill=black!30, very thick] (0, 0) rectangle (5, 5);
            \filldraw[color=black, fill=green!15, very thick] (6, 0) rectangle (7, 5);
            \filldraw[color=black, fill=black!30, very thick] (7.5, 5) rectangle (12.5, 4);
            \draw (2.5, 2.5) circle (0pt) node{$M$};
            \draw (6.5, 2.5) circle (0pt) node{$A_1$};
            \draw (10, 4.5) circle (0pt) node{$B_1$};
            \draw (0, -0.5) circle (0pt) node[anchor=west]{Compute $A_1$};
        \end{tikzpicture}    
    }
    \caption{}
\end{subfigure}
\begin{subfigure}{.45\linewidth}
    \centering
    \resizebox{!}{3cm}{
        \begin{tikzpicture}
            \filldraw[color=black, fill=black!30, very thick] (0, 0) rectangle (5, 5);
            \filldraw[color=black, fill=black!30, very thick] (6, 0) rectangle (7, 5);
            \filldraw[color=black, fill=green!15, very thick] (7.5, 5) rectangle (12.5, 4);
            \draw (2.5, 2.5) circle (0pt) node{$M$};
            \draw (6.5, 2.5) circle (0pt) node{$A_1$};
            \draw (10, 4.5) circle (0pt) node{$B_2$};
            \draw (0, -0.5) circle (0pt) node[anchor=west]{Compute B2 (solve LLSP of the same size as in step (b))};
        \end{tikzpicture}    
    }
    \caption{}
\end{subfigure}
\caption{The first updates of the matrices $A_t$ and $B_t$ 
in Eqn. (\ref{eqllsps}).}
\label{fig:alt-ref-alg}
\end{figure}
\begin{equation}\label{eqllsps}
B_{t+1}:={\rm argmin}_Y ||A_{t}Y-M||,~
A_{t+1}:={\rm  argmin}_X ||XB_{t+1}-M||.
\end{equation}

\subsection{Deterministic iterative refinement}

By incorporating  an optimal solution of generalized LLSP (cf. \cite[Ch.2]{B15}, \cite[Ch. 5]{GL13}, and the bibliography therein), we obtain the following {\bf Alg. 1.1}: 

% reconcile Q1227
\begin{equation}\label{eqllsps+}
B_{t+1} := A_t^+ M,~
 A_{t+1} := MB_{t+1}^+,~t=1,2,\dots,\tau-1.
 \end{equation}
 Here and hereafter $W^+$
denotes the {\em Moore-Penrose pseudo inverse} 
 of a matrix $W$, 
 %VP1_30_25
 and one can complete Alg. 1.1 by applying
 any   sub-algorithm for  LLSPs of (\ref{eqllsps+}).
% e.g., based on computing pseudo inverse or  QR factorization or on solving normal equations.

Our Eqn. (\ref{eqdstaur}) of Sec. \ref{scnvdtr} imposes a  restriction on  the pair of matrices  $A=A_0$ and  $M$; in particular it implies that $\sigma_{r+1}(M)<\sigma_r(M)$; the restriction becomes  severe where  $\sigma_{r+1}(M)\approx\sigma_r(M)$. %%VP
Clearly, the restriction greatly narrows the area of application of our algorithms, but they still apply to the highly important class of kernel matrices with fast decaying spectra
(see \cite{AP23,CETWa}
and the references therein).
  
 In Sec. \ref{scnvdtr} we prove  that,
under the above  restriction and for $\tau\mapsto \infty$, 
the ranges of the matrices $U_{A_{\tau}}$  and $V_{B_{\tau}}$  of the left and right  singular  vectors  of 
 $A_{\tau}$ and $B_{\tau}$, respectively,  
 %%VP 
 closely approximate  %%VP
 converge, in terms of the Principal Angle Distances between the subspaces (see Def. \ref{defpad}), to the ranges of the matrices $U^{(r)}$  
 and $V^{(r)}$ of the
left and right
singular vectors  of the matrix $M_r$ of the $r$-truncation of $M$,  respectively,
while the range of the matrix 
$U_{A_{\tau}}^TMV_{B_{\tau}}$ similarly converges to the range of the diagonal matrix 
$\Sigma^{(r)}$
of the singular values of $M_r$. Thus (\ref{eqllsps+})
enables us to   approximate (in terms of the associated subspaces) the SVD of $M_r$, an optimal LRA of $M$.  Alg. \ref{alggnritrrf} in Sec. \ref{scnvdtr} elaborates upon these computations, involving
 $O(mnr\tau)$ flops overall.
  
Iterations (\ref{eqllsps+}) are not superfast and serve just as a springboard for  our subsequent superfast refinement of LRA, but in spite of  the  limiting restriction
 on $A_0$ and $M$,
 our proof of the global convergence to optimal LRA may be of some interest  because it
 complements the previous impressive progress in proving  local convergence of the ALS method \cite{U12,ORU18,
ORU22}. 
 
 \subsection{Superfast randomized iterative refinement}
 
 {\em Johnson -- 
Lindenstrauss} randomized transform into a matrix of a smaller size enables
faster randomized  approximate solution of  generalized LLSP in (\ref{eqllsps})
 for any $t$. We first
 pre- and  post-multiply the first and the second matrix equations in (\ref{eqllsps}),
respectively, by random  sketching matrices $H_t\in \mathbb R^{n\times r}$ and  $G_t\in \mathbb R^{m\times r}$,
respectively,  and then compute approximate solutions of two generalized LLSPs of smaller sizes.
The known  efficient random  sketching algorithms  
output
a near-optimal solution  of an LLSP whp (see \cite{CDa,CFSa,EMNa}, and the bibliography therein).
 
 Furthermore, apart from the sketching stage of multiplication of
an input matrix $M\in\mathbb R^{m\times n}$ by matrices $H_t$ and $G_t$, these algorithms run  
 superfast, by using 
much fewer than $mn$ flops and
memory cells,\footnote{Hereafter $a\ll b$ shows that $a$ is much less than $b$ in context.} provided that 
$r^2\ll\min\{m,n\}$, but in  Sec. \ref{sreflsc} 
 we
run superfast  even
 the stage of multiplication by $H_t$ and $G_t$,\footnote{$H_t$ and $G_t$ are constructed as sampling and scaling matrices which are so sparse that each column contains at most one non-zero entry.}  for every $t$,
  because we solve  LLSPs  for skinny matrices $A_t\in \mathbb R^{m\times r}$
  and $B_t\in \mathbb R^{r\times n}$ and because we represent $M$ 
 by  means of
  adapting 
  the random sampling algorithm of \cite{DMM08}.
  Hence we run superfast the entire solution of both LLSPs in (\ref{eqllsps}). 
 
%VP1_30_25

 To prove  convergence of these superfast iterations to accurate LRA whp, 
  we impose stronger restriction on the pair of matrices $M$ and $A_0$; in particular we require that $2\sigma_{r+1}(M)\le \sigma_r(M)$
 and that the ratio $\frac{||M-M_r||_F} {||M-M_r||}\frac{r}{\sqrt {q-r}}$ be small for $q=\min\{m,n\}$. 
%and use the matrices of {\em random sampling and scaling} from \cite{DMM08}as our random   sketching matrices $H_t$ and $G_t$.
 Then, by non-trivially extending our
deterministic Alg. \ref{alggnritrrf}, we 
 closely approximate SVD of $M_r$ whp 
in terms of the the Principal Angle Distances between the associated subspaces (see Thm. \ref{thm:altermain}).
 Moreover, unlike  Alg. \ref{alggnritrrf},   our randomized LRA  algorithm preserves sparsity of the pair of its initial
 matrices $M$ and $A_0$. 

 Like any randomized superfast LRA algorithm, our algorithm fails with  
 a  positive constant probability  on worst case input matrices $M$ and even on the $mn+1$ matrices of the families of Example  \ref{exdlt}, but   such   ``hostile'' input
 matrices have never appeared in our tests with real world inputs, in which  initial LRAs, lying reasonably (although not very) close to optimal ones
 have been  consistently
 improved significantly 
 %%VP  and  to near-optimal level --
 already in a few iterations of  our randomized algorithm.

From \cite{DMM08}
we inherit the benefit   of outputting whp {\em  CUR LRA}, which is a special 
form of LRA, is
memory  efficient,
preserves structure of input data  (such as sparsity, non-negativity, interpretability), except for the factor $U$ of a small size, and is widely applied in data analysis, including Principal Component Analysis, multidimensional scaling, and factor analysis \cite{MD09}. 

 Together with the
  benefits we
 inherit from \cite{DMM08} the need  for fairly large  sampling of rows and columns of $M$ in order to  formally support high accuracy 
of its solution of LLSPs and LRA whp, 
  but sampling  much fewer  rows and columns turned out to be sufficient in  numerical tests with real world data
 both in \cite{DMM08} and our Sec. \ref{ststsitreflsc}.

\subsection{Related  Works}\label{srltd}

Adaptive Cross-Approximation (ACA) iterations for CUR LRA 
(cf. \cite{GTZ97,BK16,
OZ18,ALS24}
 and the references therein)  recursively output generator matrices $G$ of a
small size, from which one can readily generate CUR LRA superfast.
The iterations  monotone increase 
the volume  
$|\det(G^*G)|^{1/2}$ of a generator $G$ towards a local 
maximum. Whenever they  reach  global maximum within a factor of $h$,    
the spectral and Frobenius error norms of the output CUR LRA lie within 
a factor of $(r+1)h
\sqrt {mn}$ from those of 
optimal ones. 
 In their decades-long worldwide application  
 ACA iterations 
 consistently 
come close to such maxima superfast --  in a  reasonable  number of ACA steps -- for a large and important 
class of matrices $M$.
Adequate formal support for this excellent empirical efficiency remains a challenge,
although Cortinovis and Kressner  in \cite{CK20} proved that their deterministic  algorithm, running at superlinear cost, computes  a generator of maximal volume.
Unlike our algorithm,
ACA iterations 
involve no 
LLSPs,  use no randomization,
and to the best of our knowledge are not proved to  compute accurate
LRA under the assumptions of our Thm. \ref{thdtrmcnv}.

% \cite{PLSZa} proved an upper bounds of order $(v+1)^2$ on the spectral norm $||M-CUR||$ for any matrix $M$  and its any CUR LRA where $v := ||U|| \max\{||C||, ||R||\}$; hence uncontrolled growth of the normcan occur only where the CUR generator is ill-conditioned. This universal bound, however, greatly exceeds error norm of CUR LRA output  in practice of ACA iterations.

%In extensive numerical tests in \cite{PLSZa} such superfast combination with randomized near-optimal algorithms  of \cite{DMM08}has consistently output almost as accurate CUR LRA as \cite{DMM08}achieved at superlinear cost. Trying to obtain formal support for this algorithmwe modified the approach by combining with \cite{DMM08} the ALS method and principal angle distances rather than ACA iterations.
 The paper \cite{JNS13}  also uses the ALS method and principal angle distances but
   studies  completion of a coherent matrix\footnote{A matrix is coherent if its maximum row and column leverages scores are small in context.}   with exact rank $r$, rather than LRA.  Furthermore, it applies    
uniform element-wise sampling rather than the algorithms of \cite{DMM08}.

The papers \cite{CDa,CFSa,EMNa} and the references therein cover alternative randomized solutions of LLSPs, which,  however, do not support superfast LRA.

 Superfast LRA algorithms
\cite{MW17,BW18,
LP20,CETWa} are efficient for  some important
 special classes of matrices, different from ours; their techniques have little to do with ours.
 
  Classical  techniques
of iterative refinement are very  popular 
 for the solution of a linear system of equations (see 
\cite[Secs. 3.3.4 and 4.2.5]{S98}, 
\cite[Ch. 12 and Sec. 20.5]{H02}, \cite[Secs. 3.5.3 and 5.3.8]{GL13}) but are also  applied 
 in various other linear and nonlinear matrix computations (see \cite[page 223 and the references on page 225]{S98}). It is not clear if such techniques could be combined with the ALS method towards LRA.

\subsection{Organization of the paper}
%\noindent {\bf 1.3. Organization of the paper.}
 
In  the next section 
we cover further background
material.
In  Sec. \ref{scnvdtr}
we  analyze  deterministic iterative refinement of an LRA, which elaborates upon Alg. 1.1.
In Sec. \ref{slrasmp}  
we recall background techniques from \cite{DMM08}  for
generation of leverage scores
and sampling probabilities and their application to computing CUR LRA. 
In Sec. \ref{sreflsc} we present and analyze our randomized iterative refinement of LRA.  
In Sec. \ref{snmrtsts} we cover our numerical experiments.
We devote short Sec. \ref{sconc} to conclusions.
In Appendix A 
we prove that 
the random sampling  
solution of generalized LLSP from \cite{DMM08} adapted to our case remains near optimal whp, while becoming superfast for a large class of inputs.
%In Appendix \ref{slvscprt} we adapt  the estimates of \cite{DMM08} for leverage scores perturbation. 

%------------------------------------------------------------------------------

\section{Background}\label{sbckgr}
  
%------------------------------------------------------------------------------ 

%\subsection{$r$-top SVD, matrix norms, pseudo inverse, and Principal Angle Distance.}
%{\bf 2.1. $r$-top SVD, matrix norms, pseudo inverse, and Principal Angle Distance.}
  %\label{snrmps}
\subsection{Linear Algebra Preliminaries}\label{slap}
We assume  dealing 
with real matrices in 
$\mathbb R^{p\times q}$ 
throughout,
%VYP
 but our study can be readily extended to complex matrices.
%\footnote{Hence the Hermitian transpose %$M^*$
%is just the transpose $M^T$.} except for t%he matrices 
%of discrete Fourier transform of Section %\ref{shad}. 
Next we recall some relevant definitions,   including ones  of the introduction,
for the sake of completeness.
\begin{definition}\label{defsvd}
 \cite{GL13}.
$M=U_M\Sigma_MV^T_M$ is SVD of $M$, where
$\Sigma_M$ is the diagonal matrix of
singular values of $M$, $\Sigma_M=\diag(\sigma_j(M))_{j=1}^q$,
for $q=\min\{m,n\}$, $\sigma_j(M)=0$ for $j>\rank(M)$,
 $\sigma_j(M)$ denotes the $j$th top   singular value of $M$,
$U_M$  and $V_M$ are the matrices 
with orthonormal columns  made up  of the  left and right singular vectors of $M$, respectively.
 \end{definition}
 
 \begin{definition}\label{deftpsvd}
The $r$-{\em truncation} of $M$ is obtained by means of setting to 0 all singular values of  $M$ but its $r$ top (largest) ones. 
 The $r$-{\em top  
SVD} of  $M$ is the compact SVD (cf. \cite{GL13})
 $M_{r}=U^{(r)}\Sigma^{(r)} V^{(r)T}$ where
 the matrices  $U^{(r)}$,
$\Sigma^{(r)}$, and  $V^{(r)}$  are obtained by means of removing from $U_M$ and $\Sigma_M$ the last $m-r$ rows and from $\Sigma_M$ and $V^ T_M$ the last $n-r$ columns.
\end{definition}

$M^+:=V_M\Sigma_M^{-1} U^T_M$ is the Moore--Penrose pseudo inverse of $M$.

$||M||$ and $||M||_F$ denote  the spectral
and Frobenius norms of $M$, respectively, such that
$||M||:=\sigma_1(M)$
and $||M||_F^2:=\sum_{j=1}^{\min\{m,n\}}\sigma_j^2$.

By following \cite{HMT11} we use unified notation $|||\cdot|||$  for both of these matrix norms.

$\range(W)$ denotes the {\em range}, that is, the column span, of a matrix $W$.

%VP\begin{lemma}\label{lehg} {\rm [The norm of the pseudo inverse of a matrix product.]}
%Suppose that $A\in\mathbb R^{k\times r}$, 
%$B\in\mathbb R^{r\times l},$
%and the matrices $A$ and 
%$B$ have full rank $r\le \min\{k,l\}$. Then $|||(AB)^+|||\le |||A^+|||~|||B^+|||$. 
%\end{lemma}
 
 Applied to a subspace, the notation 
 $_{\perp}$ below defines its orthogonal complement. Applied to a matrix $M$, it defines  a matrix whose range is  the orthogonal complement of $\range(M)$.
  
\begin{definition}\label{defpad}
 \cite{JNS13}.
\label{def:principal_angle}
Let $E_1$ and $E_2$ be two subspaces of $\mathbb{R}^m$,
and let $G$, $H$, $G_{\perp}$ and $H_{\perp}$ be four matrices with orthonormal columns that generate the subspaces
$E_1$, $E_2$, and their two orthogonal
complements  in
$\mathbb{R}^m$, respectively, so 
that 
$\begin{bmatrix}
  G & G_{\perp}
\end{bmatrix}$
and
$\begin{bmatrix}
    H & H_{\perp}
\end{bmatrix}$
are orthogonal matrices.
Define 
the {\bf Principal Angle Distance} between the subspaces $E_1$ and $E_2$ (of the  same dimension) as well as between the matrices $G$ and $H$ (of the same size):
\begin{equation}
\dist(G,H):=    \dist(E_1, E_2) := ||G_{\perp}^TH||%V%_2    
    = ||H_{\perp}^TG||.%V%_2.
\end{equation}
\end{definition}

\begin{remark}\label{redst}
Let $E_1$ and
 $E_2$ be two linear subspaces of $\mathbb{R}^m$. Then

(i) $\dist(E_1, E_2)$ ranges from 0 to 1,

(ii) $ \dist(E_1, E_2) = 0$ if and only if $\textrm{Span}(E_1) = \textrm{Span}(E_2)$, 
%V Also 
and 

(iii) $\dist(E_1, E_2) = 1$ if $\textrm{rank}(E_1) \neq \textrm{rank}(E_2)$.
\end{remark}
 
%------------------------------------------------------------------------------} 
%------------------------------------------------------------------------------ 

\subsection{Optimal LRA}\label{slra}
%{\bf 2.2. $\epsilon$-rank,  LRA, and optimal LRA.}
%\label{sepslra}
%\begin{definition}\label{defrnk}
A matrix $M$ has
$\epsilon$-rank
%QL0710 \xi -> \epsilon for consistency  
%$\xi$
%P -rank 
at most $r$
 if it admits approximation within an error norm 
%$\xi$
$\epsilon$
 by a matrix $M'$ of rank at most $r$
 or equivalently if  there exist four matrices $A$, $B$, $M'$, and $E$ such that
\begin{equation}\label{eqlra}
M=M'+E,~|||E|||\le \epsilon~|||M|||,~M'=AB,~A\in \mathbb R^{m\times r},~{\rm and}~B\in \mathbb R^{r\times n}.
\end{equation}
%\end{definition}  

The $\epsilon$-rank $\rho$ of a matrix $M$ is  numerically unstable if 
$\rho$-th and $(\rho+1)$-st or $\rho$-th and 
$(\rho-1)$-st largest singular values 
of $M$ are  close to one another, but it is quite commonly used in numerical matrix computations (cf. \cite[pages 275-276]{GL13}), and we can say that a
%VP??
%------------------------------------------------------------------------------
matrix  {\em admits LRA} if it has $\epsilon$-rank $r$ where $r$ and $\epsilon$ are  small in context.

\begin{theorem}\label{thtrnc} {\rm The Eckart-Young-Mirsky theorem: LRA by means of truncation of SVD (see 
\cite{S907,EY36,M60},
\cite[Thm. 2.4.8]{GL13}).} 
The matrix $M_r$, the $r$-truncation of SVD of $M$, defines an optimal  rank-$r$ approximation of $M$ under both spectral and Frobenius norms:
\begin{equation}\label{eqmnmspfr}
 |||M_r-M|||=\min_{X:~ \rank(X)=r} |||X-M|||, 
 \end{equation}
$||M_r-M||=
 \sigma_{r+1}(M)$  and 
 $\sigma_{F,r+1}(M):=||M_r-M||_F= \sqrt{\sum_{j> r}\sigma_j^2(M)}$.
\end{theorem}

%------------------------------------------------------------------------------ 

\subsection{Hard inputs for superfast LRA}

\begin{example}\label{exdlt} {\bf  A small family of hard input matrices for superfast LRA.} Fill an $m\times n$ matrix $\Delta_{i,j}$ with 0s except for its $(i,j)$th entry filled with 1 and 
write $O_{m,n}$ to
denote the $m\times n$ null matrix, filled with 0s. Then any algorithm that misses the $(i,j)$th  
entry of its input matrix  fails to approximate within less than 1/2 one or both of the  matrices $\Delta_{i,j}$ and
$O_{m,n}$. In particular, this is the case for any  superfast algorithm. All these comments are extended to the family of $mn+1$ matrices obtained from the above $mn+1$ matrices by means of any  perturbations having  sufficiently small norms
and, with  a positive constant probability,
to randomized  algorithms.
\end{example}  
  
%------------------------------------------------------------------------------
 
\subsection{CUR LRA}\label{scur}
%{\bf 2.3. CUR LRA.}%\label{scurlra}
%------------------------------------------------------------------------------         
  $~$For $M\in \mathbb R^{m\times n}$ and two sets $\mathcal I\subseteq\{1,\dots,m\}$  
and $\mathcal J\subseteq\{1,\dots,n\}$   define
the submatrices
%\begin{equation}\label{eq111}
$M_{\mathcal I,:}:=(m_{i,j})_{i\in \mathcal I; j=1,\dots, n}$, $M_{:,\mathcal J}:=(m_{i,j})_{i=1,\dots, m;j\in \mathcal J}$, and 
$G:=M_{\mathcal I,\mathcal J}:=(m_{i,j})_{i\in \mathcal I;j\in \mathcal J}$, 
%\end{equation}
define {\em canonical rank-$r$  CUR approximation} of $M$: 
\begin{equation}\label{eqcurd} 
M\approx M'=CUR,~C=M_{:,\mathcal J},
~ U=G_r^{+},~{\rm and}~R=M_{\mathcal I,:},  
\end{equation}     and call the matrices $G$ and  $U$ its {\em generator} and  
 {\em nucleus}, respectively. 
Call $M'$ a {\em CUR LRA} of $M$  
 if $|||M'-M|||$ and $r$ are small in context.\footnote{The pioneering  
papers \cite{GZT95,GTZ97,GZT97},   define CGR  approximations with  nuclei  $G$ standing, 
say, for ``germ". In the customary acronym
CUR  ``U" can stand, say, for  
``unification factor",  but we would arrive at CNR, CCR, and CSR 
with $N$, $C$, and $S$ standing for {\em ``nucleus", ``core", and ``seed"}.} Given $G\in \mathbb R^{k\times l}$, computation of CUR LRA only involves about $kl$
memory cells and $O((k+l)kl)$ flops.

\begin{remark}
More generally, one can allow the nucleus $U$ to be any matrix in $\mathbb R^{l\times k}$; the Frobenius error norm $||M-CUR||_F$
 is minimized for 
 $U=C^+MR^+$, in which case it holds (see \cite[Sec. 4]{S99},\cite[Eqn. (6)]{MD09})  that
$||E||_F=||M-CUR||_F\le ||M-CC^+M||_F+||M-MR^+R||_F$.
We cannot compute the nucleus  $U=C^+MR^+$ superfast, however.
\end{remark}

%------------------------------------------------------------------------------

\section{Deterministic  iterative 
refinement of LRA}\label{scnvdtr}

Our next theorem implies  
 that the matrices $A_t$ in (\ref{eqllsps})
 satisfy
 $\dist(A_{t+1}, U^{(r)})\le c\cdot \dist(A_t, U^{(r)})$, for all $t$ and a constant $c < 1$. Therefore, $\dist(A_t, U^{(r)})\mapsto 0$ as $t\mapsto\infty$, under some 
 specific  
 assumptions about 
 the pair of matrices $M$ and $A_0$. One can view the proof of this theorem as warming up for our much more involved
proof of a similar result where we solve LLSPs in  
(\ref{eqllsps}) superfast whp 
by using random sampling from \cite{DMM08}.

\begin{theorem}\label{thdtrmcnv}
 Assume that
 a  matrix $A\in\mathbb{R}^{m\times r}$ has orthonormal columns,
\begin{equation}\label{eqsvdm} 
M = U^{(r)}\Sigma^{(r)}(V^{(r)})^T + U_\perp^{(r)}\Sigma_\perp^{(r)}(V_\perp^{(r)})^T
\end{equation}
is the SVD of a matrix  $M\in \mathbb{R}^{m\times n}$,
 $r < \min(m, n)$, $ 
B:= A^+M$,  
 \begin{equation}\label{eqdstaur}  
\sigma_{r+1}(M)<\sigma_{r}(M)~{\rm and}~\dist(A, U^{(r)}) 
< \sqrt{1 - \frac{\sigma^2_{r+1}(M)}{\sigma^2_r(M)}}.
\end{equation}
Then
  \begin{equation}\label{eqc}
\dist(B^T, V^{(r)}) < c \cdot \dist(A, U^{(r)})~{\rm for}~ c:= \frac{1}{\sqrt{1 - \dist(A, U^{(r)})^2}} \frac{\sigma_{r+1}(M)}{\sigma_r(M)}.
\end{equation}
\end{theorem}

\begin{remark}\label{rc}
Bounds (\ref{eqdstaur}) imply that $\frac{\sigma_{r+1}(M)}{\sigma_r(M)}<c<1$
for $c$ of (\ref{eqc}). As the ratio $\frac{\sigma_{r+1}(M)}{\sigma_r(M)}$ increases  to 1, so does $c$ as well and even does this faster than the ratio does unless $\dist(A, U^{(r)})=0$, that is (see part (ii) of Remark \ref{redst}), unless $\range(A)=\range(U^{(r)})$.
\end{remark}

\begin{proof}
We will use the following inequality:
 \begin{equation}\label{eqdstaur1} 
\sigma_r(A^TU^{(r)}) \ge \sqrt{1 - \dist(A, U^{(r)})^2},
\end{equation}
which holds because
 \begin{eqnarray*}
%V {align}
    &\big(\sigma_r(A^TU^{(r)}))^2 = \sigma_r(A^TU^{(r)}U^{(r)T}A)\\
    &= \sigma_r\big(A^T(I_m - U_{\perp}U_{\perp}^T)A\big)\\
    &= \sigma_r\big(I_r - (A^TU_{\perp})(A^TU_{\perp})^T \big)\\
    &\ge 1 - \dist(A, U^{(r)})^2. \hspace{1cm}
\end{eqnarray*}

In the latter array of relationships, the last equation holds because $A^TA=I_r$, while the  inequality  holds 
because 
the matrix $(A^TU_{\perp})(A^TU_{\perp})^T$ is Symmetric Positive Semi-Definite 
and has spectral norm $\dist(A, U^{(r)})^2$.

Since 
 $\dist(A, U^{(r)}) < 1$,  part (iii) of Remark \ref{redst} implies that 
 $\rank(A)=\rank(U^{(r)})=r$.

Ensure  that the matrix
$B := A^+M = A^TM$ has full rank $r$ by applying infinitesimal perturbation of matrix $A$ and/or $M$, then deduce (\ref{eqc}). It keeps held while we perform converse infinitesimal perturbation because
distance is a continuous function of its input matrices.

It remains to deduce 
(\ref{eqc}) in the case of a matrix $B$ of full rank.
 
Let $B= U_B\Sigma_BV_B^T$ be the SVD of $B$; then $V_B = M^TAU_B\Sigma_B^{-1}$.
Notice that 
 $$\sigma_r(B) = \sigma_r(A^TM) \ge \sigma_r(A^TU^{(r)}\Sigma^{(r)}) \ge \sigma_r(A^TU^{(r)})\sigma_r(M).$$

Combine the latter inequality with bounds (\ref{eqdstaur}) and (\ref{eqdstaur1}) and obtain
 $$||\Sigma_B^{-1}|| = \frac{1}{\sigma_r(B)} \le \frac{1}{\sqrt{1 - \dist(A, U^{(r)})^2}~\sigma_r(M)} < \frac{1}{\sigma_{r+1}(M)}.$$

Complete the proof based on the following relationships:
\begin{align}
\dist(B^T, V^{(r)}) &= ||V_B^T(V_\perp^{(r)})||\nonumber\\
&=||(V_\perp^{(r)})^T M^TAU_B\Sigma_B^{-1}||\nonumber\\
&=||(V_\perp^{(r)})^T \big(U^{(r)}\Sigma^{(r)}(V^{(r)})^T + U_\perp^{(r)}\Sigma_\perp^{(r)}(V_\perp^{(r)})^T\big)^TAU_B\Sigma_B^{-1}||\nonumber\\
&\le||\Sigma_\perp^{(r)}||\cdot ||U_\perp^{(r)}A|| \cdot ||U_B\Sigma_B^{-1}||\nonumber\\
&= \dist(A, U^{(r)})\cdot||\Sigma_\perp^{(r)}||\cdot ||\Sigma_B^{-1}||\nonumber\\
&\le \dist(A, U^{(r)}) \frac{\sigma_{r+1}(M)}{\sigma_r(B)}\nonumber\\
&\le \dist(A, U^{(r)}) \frac{1}{\sqrt{1 - \dist(A, U^{(r)})^2}} \frac{\sigma_{r+1}(M)}{\sigma_r(M)}\nonumber\\ &=c\cdot \dist(A, U^{(r)}).\nonumber\\
\label{eqn:convergent_rate}\end{align}
\end{proof}
 
\begin{remark}\label{retau}
By applying  Thm. \ref{thdtrmcnv} to matrices $M$ and $A_0$   obtain matrix $B_1$ such that $$\dist(B_1, V^{(r)}) <  c\cdot\dist(A_0, U^{(r)})$$
for $c<1$ of (\ref{eqc}).
Likewise, by applying  Thm. \ref{thdtrmcnv} to matrices 
$M^T$ and $B_1^T$ 
 obtain matrix $A_1$ such that 
$$\dist(A_1, U^{(r)}) < c\cdot\dist(B_1^T, V^{(r)}).$$
Recursively  obtain matrices $A_t$, $B_t$, $t=1,2,\dots,\tau$, such that 
$$ 
\dist(A_{t+1}, U^{(r)}) < c^2 \dist(A_t, U^{(r)})
$$
for all $t$  and $c<1$ of (\ref{eqc}).
Hence the sequence of $\dist(A_{t+1}, U^{(r)})$  converges to 0 linearly with the coefficient $c^2<1$   for $c$ of (\ref{eqc}). 
Similar argument shows that the sequence of
$\dist(B_{t+1}, V^{(r)})$  converges to 0 linearly with the  coefficient
$c^2$.
\end{remark}

%------------------------------------------------------------------------------

\begin{remark}\label{resbsbptomtr}
Let matrix $A$ only have  full rank  rather than orthonormal columns and  let 
 $U_A\Sigma_AV_A^T$ be its   SVD. Then we can
apply the proof of Thm. \ref{thdtrmcnv} to
 $U_A$ and $B':=U_A^TM$ replacing $A$ and $B$,
 respectively.
This  implies  the theorem  also for the  pair of matrices $A$ and $B$ because   
 $U_AB' = AB$ and  $\dist(B'^T, V^{(r)}) = \dist(B^T, V^{(r)})$.
The first equation is immediate because $U_AU_A^TM = AA^+M$, while the second equation holds because $B = V_A\Sigma_A^{-1}B'$ and hence $\range(B^T) = \range(B'^T)$.
Therefore, $\range(A_t)\mapsto \range(U^{(r)})$,
 and likewise  $\range(B_t)\mapsto \range(V^{(r)})$
for $A_t$ and $B_t$ of Alg. 1.1.
\end{remark}

%------------------------------------------------------------------------------
\begin{remark}\label{recnv}
%%VP
For  $\sigma_r(M) > \sigma_{r+1}(M)$, 
%%VP
the ranges of  the matrices 
$U^{(r)}$, $V^{(r)}$, $\Sigma^{(r)}
$ of (\ref{eqsvdm}) are uniquely defined, and so $\range(U_{A_t})\mapsto \range(U^{(r)})$
since $\range(A_t)\mapsto \range(U^{(r)})$. Likewise,  
$\range(V_{B_t})\mapsto \range(V^{(r)})$ since $\range(B_t)\mapsto \range(V^{(r)})$.  Hence 
$\range(U^T_{A_t}MV_{B_t})\mapsto \range(\Sigma^{(r)})$. 
Based on these observations we extend Alg. 1.1 
to approximation of the
%%VP
 ranges of the matrices  $U^{(r)}$, $\Sigma^{(r)}$, and $V^{(r)}$, representing the
SVD of $M_r$.
\end{remark}

\begin{algorithm}\label{alggnritrrf}
{\rm Deterministic iterative 
refinement of LRA.}
%------------------------------------------------------------------------------

%VYP825 I edited  below (please confirm or edit) because  reviewers said that Alg. includes no SVD except for the output !!!!!

\begin{description}
%------------------------------------------------------------------------------

\item[{\sc Input:}] 
 A positive integer $\tau$ and two matrices $M\in \mathbb R^{m\times n}$ and $A_0\in \mathbb R^{m\times r}$ such that
$0<r<\min\{m,n\}$ and  bound (\ref{eqdstaur}) holds for $A=A_0$.
% $$M \approx A_0 B_0,~
% \sigma_r(M) > \sigma_{r+1}(M), \textrm{ and}~\dist(U^{(r)}, A_0) <1.$$

\item[{\sc Computations:}]~\\
% ~ Recursively compute 
% $$ B_{t+1} := A_t^+ M \textrm{ and } A_{t+1} := MB_{t+1}^+,~t=0,1,\dots,T.$$
{\bf FOR} $t = 0, 1, \dots, \tau - 1$ {\bf DO:}
\begin{enumerate}
    \item $B_{t+1} := A_t^+ M$    \item $A_{t+1} := MB^{+}_{t+1}$
\end{enumerate}
%------------------------------------------------------------------------------

\item[{\sc Output:}]
The pair of matrices  
$A_{\tau}$, $B_{\tau}$.  
%VYP825 U_{A_{\tau}},   V_{B_{\tau}$, and \Sigma_{\tau}:=U^T_{A_{\tau}M~V_{B_{\tau}$.
\end{description}
\end{algorithm}

 Alg. \ref{alggnritrrf} computes the   matrices $A_{\tau}$ and 
 $B_{\tau}$ by using $O(mnr\tau)$
 flops; we 
   can extend this to computing 
 the triplet of matrices  $U_{A_{\tau}},   
V_{B_{\tau}}$, and 
%VYP825
$
%%VP\Sigma_{\tau}:=
U^T_{A_{\tau}}M~V_{B_{\tau}}$ by using $O((m^2+n^2)r)$ flops \cite[Fig. 8.6.1]{GL13}.
%%VP
According to Remark  \ref{recnv}, the ranges of these matrices approximate the ranges of the matrices $U^{(r)}$, $V^{(r)}$, and  $\Sigma^{(r)}$ of the $r$-top SVD of $M$.
 
By using 
 $O(mnr)$ flops one can compute Strong RRQR (Rank Revealing QR) factorization of 
$A_{\tau}$ and 
 $B_{\tau}$ instead of their SVDs \cite{GE96}; the resulting LRA would be almost as close. 

 In the next sections we will extend Alg. \ref{alggnritrrf}  to near-optimal rank-$r$ approximation of $M$ superfast whp,
 by applying random sampling and scaling of \cite{DMM08}, although, as we said in Sec. 1.1, for this extension  we need  stronger restriction on the input matrices $M$ and $A_0$.
%\end{remark}
%------------------------------------------------------------------------------

%\begin{description}
%------------------------------------------------------------------------------

%\item[{\sc Input:}] 
%Matrices $M\in \mathbb R^{m\times n}$, $A_0\in \mathbb R^{m\times r}$, 
%and $B_0\in \mathbb R^{r\times n}$ %where
%$0<r\le \min\{m,n\}$, and a positive integer $T$.
% $$M \approx A_0 B_0,~
% \sigma_r(M) > \sigma_{r+1}(M), \textrm{ and}~\dist(U^{(r)}, A_0) <1.$$

%\item[{\sc Computations:}]~\\
% ~ Recursively compute 
% $$ B_{t+1} := A_t^+ M \textrm{ and } A_{t+1} := MB_{t+1}^+,~t=0,1,\dots,T.$$
%{\bf FOR} $t = 0, 1, \dots, T - 1$ {\bf DO:}
%\begin{enumerate}
%    \item $B_{t+1} := A_t^+ M$
%    \item $A_{t+1} := MB_{t+1}^+$
%\end{enumerate}
%------------------------------------------------------------------------------
%------------------------------------------------------------------------------

\section{Leverage Scores, Sampling and Scaling Matrices, LLSPs, and CUR LRA}\label{slrasmp}   

%----------------------------------------- 
  
\subsection{Overview}  
  
 To accelerate deterministic  
 iterative refinement   
 of the previous section,
we  compress 
its generalized  LLSPs by applying to them  random  sampling of \cite{DMM08}. Namely, for a fixed integer $t$, $0\le t\le \tau-1$, we successively 
compute (i) leverage scores defined by the  singular spaces  of the matrix $A_t$, (ii) 
sampling probabilities  defined by the  leverage scores (see the next two subsections),
(iii)
the sampling and scaling matrix $\mathcal S$ 
(see  Sec. \ref{ssrcs}), and (iv) the solution $B_{t+1}$ of  
 the compressed generalized  LLSP (of a smaller size)
defined by the pairs of matrices $\mathcal SA_t$ and $\mathcal SM$. 

 In Sec. \ref{serrllsp} we
 prove that whp $B_{t+1}$ is a near-optimal  solution  of the generalized LLSP
 for $M$ and $A_t$.  
 We can  readily extend this study  to
 obtain whp a near-optimal solution $A_{t+1}$ of generalized LLSP for the pair $M$ and $B_{t+1}$. 
 
 In the next section we apply our solution of generalized LLSPs recursively to arrive at randomized implementation of 
  iterative refinement   (\ref{eqllsps}).  
Fig. \ref{fig:llsp-samp} shows a single step of the above recursive process, for $A_t$ represented  with $A$.  

 The computation of 
leverage scores
is the bottleneck 
of the randomized compression LLSP algorithm of \cite{DMM08}, applied to an $m\times n$ matrix  $M$,
but this computation is simplified in our case because we only apply it for LLSPs with matrices $A_t$ and $B_{t}$ of smaller sizes.
 
 \begin{figure}[hbt!]
\begin{subfigure}{.45\linewidth}
    \centering
    \resizebox{!}{3cm}{
        \begin{tikzpicture}
            \filldraw[color=black, fill=black!5, very thick] (0, 0) rectangle (2.5, 5);
            \filldraw[color=black, fill=black!5, very thick] (3.5, 0) rectangle (4.5, 5);
            \filldraw[color=black, fill=black!5, very thick] (5.5, 5) rectangle (8, 4);
            \draw (1.25, 2.5) circle (0pt) node{$M$};
            \draw (3, 2.5) circle (0pt) node{$-$};
            \draw (4, 2.5) circle (0pt) node{$A$};
            \draw (6.75, 4.5) circle (0pt) node{X};
            \draw (0, -0.5) circle (0pt) node[anchor=west]{Solve $Y = \argmin_X ||M - AX||_F$};
            
        \end{tikzpicture}
    }
    \caption{}
\end{subfigure}
\begin{subfigure}{.45\linewidth}
    \centering
    \resizebox{!}{3cm}{
        \begin{tikzpicture}
            \filldraw[color=black, fill=black!30, very thick] (-4, 2.5) rectangle (-1.5, 5);
            \draw (-3, 3.75) circle (0pt) node{$M'$};
            
            \draw (-1, 4) circle (0pt) node{$=$};
            
            \filldraw[color=black, fill=green!15, very thick] (-0.5, 2.5) rectangle (4.5, 5);
            \draw (2, 3.75) circle (0pt) node{$\mathcal{S}$};

            \fill (5, 4) circle (2pt);
            
            \filldraw[color=black, fill=black!5, very thick] (5.5, 0) rectangle (8, 5);
            \draw (6.5, 3.75) circle (0pt) node{$M$};
            
            \draw (-4, -0.5) circle (0pt) node[anchor=west]{Construct sampling and scaling matrix $\mathcal{S}$ and obtain $M'$};
        \end{tikzpicture}
    }
    \caption{}
\end{subfigure}

\medskip

\begin{subfigure}{.45\linewidth}
    \centering
    \resizebox{!}{3cm}{
        \begin{tikzpicture}
            \filldraw[color=black, fill=black!30, very thick] (-2.5, 2.5) rectangle (-1.5, 5);
            \draw (-2, 3.75) circle (0pt) node{$A'$};
            
            \draw (-1, 4) circle (0pt) node{$=$};
            
            \filldraw[color=black, fill=green!15, very thick] (-0.5, 2.5) rectangle (4.5, 5);
            \draw (2, 3.75) circle (0pt) node{$\mathcal{S}$};

            \fill (5, 4) circle (2pt);
            
            \filldraw[color=black, fill=black!5, very thick] (5.5, 0) rectangle (6.5, 5);
            \draw (6, 3.75) circle (0pt) node{$A$};
            
            \draw (-4, -0.5) circle (0pt) node[anchor=west]{Use the same sampling and scaling matrix $\mathcal{S}$ and obtain $A'$};
        \end{tikzpicture}
    }
    \caption{}
\end{subfigure}
\begin{subfigure}{.45\linewidth}
    \centering
    \resizebox{!}{3cm}{
        \begin{tikzpicture}
            \filldraw[color=black, fill=black!30, very thick] (0, 2.5) rectangle (2.5, 5);
            \filldraw[color=black, fill=black!30, very thick] (3.5, 2.5) rectangle (4.5, 5);
            \filldraw[color=black, fill=black!5, very thick] (5.5, 5) rectangle (8, 4);
        
            \draw (1, 3.5) circle (0pt) node{$M'$};
            \draw (4, 3.5) circle (0pt) node{$A'$};
            \draw (6.75, 4.5) circle (0pt) node{$X$};
            \draw (3, 3.5) circle (0pt) node{$-$};
            \draw (-1, 0.4) circle (0pt) node[anchor=west]{Compute sub-sampled problem $Y' = \argmin_X ||M' - A'X||_F$ };
            \draw (-1, -0.6) circle (0pt) node[anchor=west]{$Y'$ is an approximation of $Y$ in the sense that $||M-AY'||_F \approx ||M-AY||_F$};
        \end{tikzpicture}    
    }
    \caption{}
\end{subfigure}
\caption{Part (a)
displays the original problem  of solving LLSP (defined by two matrices                                                                                                                                                                                                                                                                                                                         $A$ and $M$). Parts (b) and (c) show compression of these   matrices into two matrices  $M'$ 
and $A'$ of  smaller sizes.
Part (d) displays 
compressed problem
defined by the pair of matrices $M'$ 
and $A'$.}
\label{fig:llsp-samp}
\end{figure}

\subsection{Leverage Scores and Sampling Probabilities: Definition}\label{slevscr}

 Leverage scores
are defined by
left (respectively, right) singular spaces of the matrices $A_t$ (respectively, $B_t$) of (\ref{eqllsps});  in Remark \ref{rescrtoprb}  we extend their  computation to 
sampling probabilities.
Unlike \cite{DMM08}
 neither scores nor probabilities depend on the matrix $M$ in our case.

Next we unify the definition of leverage scores
 by writing $W$ for
$A_t$ and $B_t$ and for all
$t=0,1,2, \dots,\tau-1$.

\begin{definition}\label{def:leverage_score}
(See \cite{DMM08}.)
% Let $M$ be 
Given a ${g\times h}$ matrix $W$,
with $\sigma_r(W) > \sigma_{r+1}(W)$, and its 
$r$-top SVD  $W_{r}=U^{(r)}\Sigma^{(r)} V^{(r)T}$,
write 
%V Define the rank-$r$ {\bf Column Leverage Scores} of $A$ 
\begin{equation}\label{eqsmpl}
    \gamma_i := \sum_{j = 1}^{r} v_{i, j}^2, ~\textrm{ for } i = 1,
    \dots,h,~{\rm and}~
\end{equation}
%and similarly the rank-$r$ {\bf Row Leverage %V Scores} 
\begin{equation}
    \Tilde \gamma_i: = \sum_{j = 1}^{r} u_{i, j}^2, ~\textrm{ for } i = 1,\dots,g,
    \end{equation}
where $v_{i, j}$ and $u_{i, j}$ denote the $(i, j)$-th entry of 
matrices $V^{(r)}$ and $U^{(r)}$, respectively. 
Then call $\gamma_i$ and $\Tilde \gamma_i$ the rank-$r$ {\bf Column} and {\bf Row Leverage Scores}  of $W$, respectively.
\end{definition}

\begin{remark} The 
 $r$-top left and right singular spaces of $W$
 are uniquely defined
 if $\sigma_r(W) > \sigma_{r+1}(W)$. In particular this holds where $W$ is a rank-$r$ matrix $A_t$
 or $B_t$.
 \end{remark}

\begin{remark}\label{rescrtoprb} The row/column leverage scores scaled by $r$  define  probability distributions because $\sum_{i = 1}^m \Tilde \gamma_i = \sum_{i=1}^n \gamma_i = r$. 
%VYPIn fact 
 By following \cite{DMM08}  we allow soft probability distributions   $\{p_i|i = 1,\dots,n\}$ and $\{\Tilde{p}_i|i=1,\dots,m\}$: given $\{\gamma_i|i = 1,\dots,n\}$  and $\{\Tilde{\gamma}_i|i=1,\dots,m\}$, fix $\beta$, $0<\beta\le 1$, allow  
to increase sampling size by a factor of $\beta^{-1}$, and only 
 require that  
 \begin{equation}\label{eqlevsc}
 p_i > 0,~ 
 p_i \ge \beta\gamma_i/r
~{\rm for}
~i=1,\dots,n,~{\rm and}~
\sum_{i=1}^np_i=1, 
\end{equation}
\begin{equation}\label{eqlevsc_row}
 \Tilde p_i > 0,~ 
 \Tilde p_i \ge \beta\Tilde \gamma_i/r
~{\rm for}
~i=1,\dots,m,~
{\rm and}~
\sum_{i=1}^m\Tilde p_i=1,
%V.   
\end{equation}
\end{remark}

\begin{remark}\label{rescrprbsprf} 
The computation of the
SVDs of the matrices $A_t$ and $B_t$ involves $O((m+n)r^2)$ flops; it is superfast for $r^2\ll \min\{m,n\}$,
and then so 
is the above computation of
leverage scores and probabilities,
unlike such computations in \cite{DMM08}, involving the SVD of $M$. 
\end{remark} 

%------------------------------------------------------------------------------

\subsection{Sampling and Scaling Matrices}\label{ssrcs}

\cite[Algs. 4 and 5]{DMM08} define
sampling and scaling for a matrix $M\in \mathbb R^{m\times n}$. Next we adapt them  to a matrix 
$A\in \mathbb R^{m\times r}$.
We omit straightforward extension of this adaptation to 
a matrix 
$B\in \mathbb R^{n\times r}$.  

\begin{algorithm}\label{algsmplex} (\cite[Alg. 4]{DMM08},
{\rm The Exactly($l$) Sampling and Scaling.})

%------------------------------------------------------------------------------

\begin{description}
 
%------------------------------------------------------------------------------
 
\item[{\sc Input:}] 
Two integers $l$ and $m$ such that $1\le l\le m$ and $m$ 
% QL0920
positive scalars $\tilde p_1,\dots,\tilde p_m$
such that $\sum_{i=1}^m\tilde p_i = 1$.

%------------------------------------------------------------------------------

\item[{\sc Initialization:}]
Let $\widehat S:=O_{m,l}$ and $\widehat D:=O_{l,l}$ be 
 matrices filled with zeros. 

\item[{\sc Computations:}] ~ 
 
%(1)  
{\bf FOR} $t = 1,\dots,l$ {\bf DO}

\begin{enumerate}
\item Generate 
%VYP independently
%QL-LAA-2
%We are generating a random every iteration, so it is usually assumed,
%but it probably would not hurt to mention that generation is an independent event
an independent random integer $N$ from $\{1,\dots,m\}$ such that $\textrm{Prob}\{N = i\} = \tilde p_i$;
\item Set $\widehat s_{N, t}: = 1$, {\it i.e.}, set the $N$-th element of the $t$-th column of $\widehat S$ to 1;
\item Set $\widehat d_{t, t}: = 1/\sqrt{l\tilde p_N}$;
\end{enumerate}

{\bf END FOR}

\medskip

%(2) Write $s_{i,t}=0$ for all pairs of $i$ and  $t$ unless $i=i_t$.

%------------------------------------------------------------------------------

% QL0126  S is a row sampling/scaling matrix of A of size $m\times r$,
% so it is natural for S to have size $l\times m$
% we obtain this by taking transpose of $\widehat S\widehat D$
% In other words $\widehat S\widehat D$ is for sampling columns (multiply from the right),
% the transpose is for sampling rows (multiply from the left)

\item[{\sc Output:}] $l\times m$ sampling and scaling matrix $\mathcal S:=(\widehat S\widehat D)^T$ for $m\times l$ sampling matrix $\widehat S=(\widehat s_{i,t})_{i,t=1}^{m,l}$ and
$l\times l$ scaling matrix 
$\widehat D=\diag(\widehat  d_{t,t})_{t=1}^l$.

\end{description}
\end{algorithm}
%QL0217 scurlra is an undefined reference, changing it to remark \ref{relratocur}
Except for 
%Appendix \ref{scurlra}
Remark \ref{relratocur}
we will only use  sampling matrices $\widehat S$
and  scaling matrices $\widehat D$ implicitly  -- represented by their product $\mathcal S$.

\begin{remark}\label{relsccmpl}
The algorithm involves $l$ multiplications,  
$l$ divisions, and
computes
$l$ square roots for $l=O(r^2\log(r)/\epsilon^2)$. Thus, the algorithm is superfast provided that $r^2\log(r)\ll mn$ and $\epsilon>0$ is a constant.
 \end{remark}

\noindent \cite[Alg. 5 (Expected($l$))]{DMM08}, an alternative %randomized option
to Alg. \ref{algsmplex}, computes 
%V?an $n\times g$ 
random sampling  and  scaling matrices. Their size is
   randomly determined  and is expected to be much smaller than in Alg. \ref{algsmplex}. This local acceleration, however, little affects the overall complexity of our  refinement of LRA,
which depends on the strong lower bound  on $r$ in Remark \ref{rescrprbsprf},
and we 
simplify our presentation  by only using Alg. \ref{algsmplex}.

\subsection{Randomized Solution of Generalized LLSP: Output Errors}\label{serrllsp}

Instead of costly  computation of optimal solution $X:=A^+M$
to the  
generalized LLSP
$$X={\rm argmin}_Y||AY-M||,$$ we  faster compute optimal solution
 $\Tilde{X} := (SA)^+SM$ to the compressed problem
 $$\Tilde{X}={\rm argmin }_Y||\mathcal SAY-\mathcal SM||$$ 
 for a sampling and scaling matrix $\mathcal S$.
 The following theorem, adapted from  \cite[Thm. 5]{DMM08} and proved in Appendix \ref{appendix:pfdmm08thm5}, shows that an
optimal solution of the compressed problem is  near-optimal for the original one.

\begin{theorem}\label{thm:dmm08thm5}
Let  
%Vbe a rank $r$ matrix, 
$\Tilde \gamma_i$ for $i = 1,\dots,m$ be the rank-$r$ row leverage scores of a rank-$r$ matrix $A\in\mathbb{R}^{m\times r}$
  and let $M\in\mathbb{R}^{m\times n}$.
  %V be any matrix. 
  Fix three
  positive numbers
$\epsilon < 1/2$, $\xi < 1/3$,
%V
 and $\beta \le 1$
 and
compute probability distribution $\{\Tilde p_i| i=1,\dots,m\}$ satisfying 
%V the 
relationships 
%V of  
(\ref{eqlevsc_row}).
%V, Let
Write~\footnote{ The large constant factor 1296 in the estimates throughout the paper comes from \cite{DMM08},
but in the tests with real world inputs   both in \cite{DMM08} and our  paper much smaller row and column samples were sufficient. E.g., in our tests in Sec. \ref{ststsitreflsc}  we succeeded by sampling just $15r$ rows and columns.} 
\begin{equation}\label{eql}
l: = \lceil1296\beta^{-1} r^2\epsilon^{-2}\xi^{-4}\rceil \end{equation}
 and let $\mathcal S$ be the 
sampling and scaling matrix  
 output by Alg. \ref{algsmplex}. 
Then 
%VP(see Fig. 2)
\begin{equation}
    \textrm{rank}(SA)= r \hspace{0.5cm} \textrm{and} \hspace{0.5cm} ||A\Tilde{X} - M||_F \le (1 + \epsilon)||AA^+M - M||_F
     \label{eqn:error_bound}
\end{equation}
with a probability no less than $1 - \xi$
where
\begin{equation}\label{eqrndllsp}
    \Tilde{X} := (\mathcal SA)^+\mathcal SM.
\end{equation}
\end{theorem} 

%V By performing sampling 
%QL0920 added comments based on reviewer's comment
We apply this theorem to the matrices $A=A_t$ in our iterative LRA algorithm and can immediately  extend it to the matrices $B_t$. 

According to that theorem, whp $\tilde X$ of (\ref{eqrndllsp}) is a near-optimal solution of generalized   LLSP, computed superfast in \cite{DMM08}
provided that  leverage scores are known.

\section{Iterative refinement of LRA: an algorithm and analysis}\label{sreflsc}
 
 In this section we apply our solution of generalized LLSPs recursively to arrive at a superfast randomized implementation of 
  iterative refinement   (\ref{eqllsps}) of LRA.  
  
  We prove in our Main Theorem  that whp $\range(A_t)\mapsto
 \range(U^{(r)})$,
 and similarly
 $\range(B_t)\mapsto
 \range(V^{(r)})$
 as $t\mapsto\infty$,
 where $M=U^{(r)}\Sigma^{(r)} V^{{(r)}T}$
 denotes the SVD of $M_r$.
 
 As in Remark \ref{resbsbptomtr}, we can readily extend 
 the above results to the convergence whp 
%VP 
 $\range(U_{A_t})\mapsto \range(U^{(r)})$, $\range(V_{B_t})\mapsto \range(V^{(r)})$, and  $\range(U_{A_t}^TMV_{B_t})\mapsto \range(\Sigma^{(r)})$ and hence to approximation whp 
  of the SVD of $M_r$
  %%VP  
  in terms of the ranges of the associated matrices, 
 but we only need this extension for $t=\tau$. Furthermore, we obtain an LRA without explicitly multiplying $M$ by $U_{A_t}^T$
 or $V_{B_t}$.????????????????????????????????????????????????????
{\bf Qi, I hope this paragraph is correct, but if you have extra time, please double check.}

\subsection{The algorithm} 

Any  solution of LLSP, e.g., by means computing Moore-Penrose pseudo inverse, can be used at stages 3 and 6 of our next algorithm.

\begin{algorithm}\label{algalter} 
{\rm [Iterative  Refinement of LRA by Using Sampling Probabilities from \cite{DMM08}.]} 

%------------------------------------------------------------------------------

\begin{description}
  
%------------------------------------------------------------------------------

\item[{\sc Input:}] Two matrices
$M\in \mathbb R^{m\times n}$ and 
$A_0\in \mathbb R^{m\times r}$
for $1\le r\le \min\{m,n\}$, a positive integer $\tau$, and two positive numbers $\epsilon$ and $\xi < 1$.

\item[{\sc Initialization:}]
Fix a real $ \beta$, $0< \beta \le 1$.
%, and then apply Eqns. (\ref{eqlevsc}) and (\ref{eqlevsc_row}).
%-----------------------------------------------------------------------------
%%%Qi editing Alg. 5.1 so that it is consistent with the figure and Alg. 3.1
\item[{\sc Computations:}]~\\
{\bf FOR} $t = 0, 1,\dots, \tau-1$ {\bf DO:}

\begin{enumerate}
\item%1 
Compute the row leverage scores $\tilde \gamma_i$ of $A_t$ and probabilities $\tilde p_i$ satisfying (\ref{eqlevsc_row}) for $i = 1,\dots, m$. 
\item%2
Apply Alg. \ref{algsmplex} 
%VP to matrices $M$ and $A_t$ 
for $l = 1296\beta^{-1} r^2\epsilon^{-2}\xi^{-4}$ and probabilities $\tilde p_i$ to
compute  sampling   and  scaling   
matrix $\mathcal S$.
\item%3 
Compute $B_{t+1} = (\mathcal SA_t)^+\mathcal SM$.
\item%4
Compute the column leverage scores $ \gamma_i$ of $B_{t+1}$
 %VP? , fix an appropriate $0< \beta \le 1$,
  and  
probabilities  $p_i$ satisfying (\ref{eqlevsc}) for $i = 1,\dots, n$. 
 \item%5
Apply Alg. \ref{algsmplex}  for $l = 1296\beta^{-1} r^2\epsilon^{-2}\xi^{-4}$ and  probabilities $ p_i$ replacing $\tilde p_i$ to
 update the sampling   and   scaling  
matrix $\mathcal S$. 
\item%6
Compute $A_{t+1} = M\mathcal S^T(B_{t+1}\mathcal S^T)^+$.

\end{enumerate} 
{\bf END FOR}

%VP \begin{enumerate}
%\item[] Apply Alg. \ref{algsmplex} to matrices $M$ and $A_\tau$ for
%$l = 1296\beta^{-1}r^2\epsilon^{-2}\xi^{-4}$ to
%compute  sampling   and  scaling matrices $S$ and $D$.
%\item[] Compute $B_{\tau+1} = (D^TS^TA_\tau)^+D^TS^TM$.

%\end{enumerate}

\item[{\sc Output:}] $A_{\tau}$ and $B_{\tau}$.
%------------------------------------------------------------------------------

\end{description}
\end{algorithm}

The algorithm is superfast provided that $\tau r^2\ll \min\{m,n\}$ because the computational  cost of its $t$th iteration for every $t$ is dominated by the cost of generation of the leverage scores
(which is sublinear for $r^2\ll\min\{m,n\}$ according to  Remarks \ref{rescrprbsprf}
and \ref{relsccmpl})  
and of multiplication of $M$ by $\mathcal S$ and $\mathcal S^T$, reduced to
selection and scaling of $l$ rows and $l$ columns of the matrix $M$, while we choose $l=O(r^2)$.

%VP1_30_25 (see the rest of this section). 
%QL0206
\begin{remark}\label{relratocur}
The algorithm outputs $A_\tau$ and $B_\tau$ such that 
$A_\tau B_\tau \approx M$ but can be readily modified to output a CUR approximation of $M$.
Indeed, at step 2  of the last iteration, for $t=\tau-1$,
we only output  sampling and scaling matrix $\mathcal{S}$
and update it at step 5 but actually we first compute
 scaling matrices $D_1$ and $D_2$ and sampling matrices $S_1$ and $S_2$ 
and then obtain $\mathcal{S} = D_1S_1$ in step 2 and $\mathcal{S} = D_2S_2$ in step 5.
Now let us reuse both sampling and both scaling matrices and obtain
$$
B_\tau = (D_1S_1A_{\tau-1})^+D_1S_1M = (D_1S_1A_{\tau-1})^+D_1R 
$$
and
$$
A_\tau = MS_2^TD_2^T(B_\tau S_2^TD_2^T)^+ = CD_2^T(B_\tau S_2^TD_2^T)^+,
$$
for column and row submatrices  $C$ and $R$ of $M$, respectively (possibly, with repeated columns and rows).
Then $A_\tau  B_\tau = CUR$ for $U = D_2^T(B_\tau S_2^TD_2^T)^+(D_1S_1A_{\tau-1})^+D_1$.
\end{remark}

\subsection{Linear convergence: basic estimates}\label{appendix:pfbasic}

Next we specify  the
principal angle distances in   Alg. \ref{algalter}.

\begin{theorem}\label{thm:contracting_distance}
Let 
$$M = 
\begin{bmatrix}U^{(r)} & U_{\perp}\end{bmatrix}
\begin{bmatrix}\Sigma^{(r)} & \\ & \Sigma_{\perp} \end{bmatrix}
\begin{bmatrix}V^{(r)T} \\ V_{\perp}^T \end{bmatrix}$$
be the SVD of $M\in\mathbb R^{m\times n}$ where $r\le \min\{m, n\}$, let  $\sigma_r(M) > \sigma_{r+1}(M)$, and let $A$ be an $m\times r$ matrix  such that
\begin{equation}\label{eqn:con_dist1}
   \delta:=\dist(A, U^{(r)})  < 1.
\end{equation} 
Fix three positive numbers $\epsilon < 1/2$, $\xi < 1/3$, and $\beta \le 1$ 
and compute the rank-$r$ row leverage scores $\{\gamma_i| i=1,\dots,m\}$ of $A$
and a sampling probability  distribution $\{p_i| i = 1,\dots,m\}$ satisfying (\ref{eqlevsc_row}).
Suppose that 
Algorithm \ref{algsmplex} applied
for $l = 1296\beta^{-1} r^2\epsilon^{-2}\xi^{-4}$
outputs a sampling and scaling  matrix 
$\mathcal{S}$.
Write $B: = (\mathcal{S}A)^+\mathcal{S}M$.
Then 
\begin{equation}\label{eqn:con_dist2}
    \dist(B^T, V^{(r)}) \le c \cdot \dist(A, U^{(r)}) + e,
\end{equation}
with a  probability at least $1 - \xi$
for 
$$c :=\frac{1}{\sqrt{1 - \delta^2}}\cdot \frac{\sigma_{r+1}(M)}{\sigma_r(M)}
~\textrm{and}~ 
e := 2c\frac{\sigma_{F, r+1}(M)}{\sigma_{r+1}(M)}\sqrt{\epsilon}\le 2c\sqrt{(\rank(M)-r)\epsilon}.$$
\end{theorem}

\begin{proof}
Let $U_A\Sigma_AV_A^T=A$ be a compact SVD where $\Sigma_A, V_A\in\mathbb{R}^{r\times r}$; then 
 recall that  $\begin{bmatrix} U_A & (U_A)_{\perp} \end{bmatrix}$ is an orthogonal matrix 
and therefore
$\dist(A,U^{(r)}) = 
||(U_A)_{\perp}^TU^{(r)}|| = 
||U_A^TU_{\perp}||$.

By virtue of Thm. \ref{thm:dmm08thm5}, $\textrm{rank}(\mathcal{S}A) = r$ with a probability no less than $1 - \xi$, 
and we assume that $\mathcal{S}A$ has full rank for the rest of the proof.
% Assume that the matrix $B := (\mathcal{S}A)^+\mathcal{S}M$ has full rank.

Now write  $B' := \Sigma_A^{-1}V_A^TB$
and, since $\Sigma_AV_A^T$ is a $r\times r$ full rank matrix, obtain
$$ B' = \Sigma_A^{-1}V_A^T(\mathcal{S}A)^+\mathcal{S}M=(\mathcal{S}U_A)^+\mathcal{S}M,~~
%VP1_30_25 ~\textrm{ and }~ 
 \dist(B^T, V^{(r)}) = \dist(B'^T, V^{(r)}),
$$
$$ AB = A(\mathcal{S}A)^+\mathcal{S}M = U_A^T(\mathcal{S}U_A^T)^+\mathcal{S}M = U_AB'.
$$
Then there exists a 
QR factorization of $B'$ such that 
$$    B' = RQ^T ~\textrm{ and }~ Q^T = R^{-1}B' \in\mathbb{R}^{r\times n}. 
$$ 

Express the matrix $(SU_A)^+S$
as follows:
\begin{equation}
    (\mathcal{S}U_A)^+\mathcal{S} = \begin{bmatrix}C_1 & C_2\end{bmatrix} \cdot \begin{bmatrix}U_A^T \\ (U_A)_{\perp}^T \end{bmatrix}, \label{eqn:lin_comb_sgs}
\end{equation}
for a unique pair  of matrices $C_1$ and $C_2$.
Deduce that 
\begin{eqnarray*}
%V{align}
   & \dist(B'^T, V^{(r)})= ||Q^TV_{\perp}||
   %V%_2
   \\
                &= ||R^{-1}(\mathcal{S}U_A)^+\mathcal{S}MV_{\perp}||%V%_2
                \\
                &= ||R^{-1}(\mathcal{S}U_A)^+\mathcal{S}U_{\perp}\Sigma_{\perp}||%V%_2
                \\\label{eqn:sgsu}
                &\le ||R^{-1}||~~%V%_2 
                ||(C_1U_A^T + C_2(U_A)_{\perp}^T) U_{\perp}\Sigma_{\perp}||
               %V%_2
                \\ 
                &\le \frac{1}{\sigma_r(B')} \big( ||C_1(U_A)^TU_{\perp}\Sigma_{\perp}||%V% _2
                 + ||C_2(U_A)_{\perp}^TU_{\perp}\Sigma_{\perp}||%V%_2
                 \big). 
\end{eqnarray*}
 
 Next deduce  bound (\ref{eqn:con_dist2}) 
provided that the following properties hold 
with a probability no less than $1-\xi$: 

(1) $C_1 = I_r$,  

(2) $||C_2(U_A)_{\perp}^TU_{\perp}\Sigma_{\perp}||  
\le 2\sqrt{\epsilon} ||\Sigma_{\perp}||_F = 2\sqrt{\epsilon}\sigma_{F, r+1}(M)$, and 

(3) $\sigma_r(B') \ge \sqrt{1 - \delta^2}~\sigma_r(M)$.

 Indeed,
\begin{align*}
\dist(B'^T, V^{(r)}) &\le \frac{1}{\sigma_r(B)}\cdot\Big(
||U_A^TU_{\perp}||\cdot||\Sigma_\perp|| +  ||C_2(U_A)_{\perp}^TU_{\perp}\Sigma_{\perp}||
\Big) \\
&\le \frac{\sigma_{r+1}(M)}{\sqrt{1 - \delta^2}\sigma_r(M)}\cdot \dist(A, U^{(r)})
    +
\frac{2\sqrt{\epsilon}\sigma_{F, r+1}(M)}{\sqrt{1 - \delta^2}\sigma_r(M)}\\
&= c \cdot \dist(A, U^{(r)}) + e.
\end{align*}

It remains to prove properties  (1) -- (3)
above. 
\medskip

\noindent {\bf Property (1):} 
 Eqn. (\ref{eqn:error_bound})
implies that the matrices $\mathcal{S}A$ and hence $\mathcal{S}U_A$ have full rank $r$. Hence
$$    (\mathcal{S}U_A)^+\mathcal{S}U_A = I_r,
$$
while by definition
\begin{align*}
(\mathcal{S}U_A)^+\mathcal{S}U_A &= (C_1U_A^T + C_2(U_A)_{\perp}^T)U_A\\
&= C_1U_A^TU_A + C_2(U_A)_\perp^TU_A\\
&= C_1.
\end{align*}    
\medskip
\noindent {\bf Property (2):} 
Consider the following generalized LLSP,
   $$\min_{X} ||Y - U_AX ||_F$$
where  $Y = (U_A)_{\perp}(U_A)_{\perp}^TU_{\perp}\Sigma_{\perp}$ denotes 
an $m\times (n-r)$ matrix.
Clearly,  $\min_{X} ||Y - U_AX ||_F = ||Y||_F$ 
because  
 $\range(Y)$ is orthogonal to  $\range(U_AX)$.

Furthermore, recall that  $\range(Y)$ and $\range(A)$ are  
orthogonal to one another. Combine this observation with 
%Equation 
 (\ref{eqn:lin_comb_sgs}) and deduce that
 \begin{eqnarray*}
%V {align}
    &||Y - U_A(\mathcal{S}U_A)^+\mathcal{S}Y||_F^2\\
    =& || Y - U_A(C_1U_A^T + C_2(U_A)_{\perp}^T)Y||_F^2\\
    =& || Y - U_AU_A^TY - U_AC_2(U_A)_{\perp}^TY||_F^2\\
    =& ||Y||_F^2 + ||U_AC_2(U_A)_{\perp}^TY||_F^2.
\end{eqnarray*} 
%V {align}
%V This is true again 
%V due to the 
%V because the column space of $Y$ 
%V being 
%V is orthogonal to the column space of $G$. 
Recall from 
%Equation
 (\ref{eqn:error_bound}) that $$||Y - A(\mathcal{S}A)^+\mathcal{S}Y||_F^2 \le (1+\epsilon)^2||Y||_F^2$$ and conclude that $$||C_2(U_A)_{\perp}^TY||_F < \sqrt{2\epsilon + \epsilon^2}||Y||_F < 2\sqrt{\epsilon}||\Sigma_{\perp}||_F.$$
%Vic\\
\medskip
\noindent {\bf Property (3):}
Recall that $B' = (\mathcal{S}U_A)^+\mathcal{S}M$, and therefore
%QL fixed typos: changed k to r
\begin{eqnarray*} 
%V {align}
   & \sigma_r(B') = \sigma_r\big( (U_A^T + C_2(U_A)_{\perp}^T)M \big)
???????????????????????????????????????????????????     
   \\
    &\ge \sigma_r(U_A^TM)\\
    &\ge \sigma_r(U_A^TU^{(r)}\Sigma^{(r)})\\
    &\ge \sigma_r(U_A^TU^{(r)}) \cdot \sigma_r(M).
\end{eqnarray*}
%V {align}

Combine this bound on $\sigma_r(B')$
with (\ref{eqdstaur1}) and obtain
 that $\sigma_r(B') \ge \sqrt{1 - \delta^2}~\sigma_r(M)$,  which implies that $\textrm{rank}(B') = r$.
\end{proof}

%VYP1 

\subsection{Linear convergence under weaker assumptions}\label{appendix:pf}

\noindent To simplify notation, write 
$\sigma_j:=\sigma_j(M)$ for $j=1,2,\dots,r$ and $\Bar\sigma_{r+1}:= \sigma_{F, r+1}(M) = ||M - M_r||_F$. 

\noindent Our next goal is  the bound
$\dist(B^T, V^{(r)}) < c'\cdot \dist(A, U^{(r)})$ for a constant $c' < 1$ under a small lower bound on $\dist(B^T, V^{(r)})$ and  under the additional bounds $\delta<1/2$ and $\sigma_r< 2\sigma_{r+1}$, but we will begin
with  weaker upper bounds on that distance   and $\delta$.

\begin{lemma}\label{lemma:alt_ref}
Let $m, n, r, \epsilon, \delta$, $M$, $U^{(r)}$, $V^{(r)}$, $A$, $B$, $c$, and $e$ be defined as in Thm. \ref{thm:contracting_distance} such that (\ref{eqn:con_dist1}) and (\ref{eqn:con_dist2}) hold.
  Furthermore, let $\frac{\sigma_{r+1}}{\sigma_r}<\sqrt{1-\delta^2}$
for $\delta:=\dist(A,U^{(r)})$ of (\ref{eqn:con_dist1}).
Then
$$
\dist(B^T, V^{(r)}) <\max\Big\{ \dist(A,U^{(r)}),
 \frac{2\Bar\sigma_{r+1}}{\sigma_r\sqrt{1-\delta^2} - \sigma_{r+1}}\sqrt{\epsilon}\Big\}.
$$
\end{lemma}

\begin{proof}
Notice that
\begin{align}
    e &= \frac{2\Bar\sigma_{r+1}}{\sqrt{1-\delta^2}~\sigma_r}\sqrt{\epsilon}\nonumber\\
    &= \frac{1}{\sqrt{1-\delta^2}}\cdot\frac{\sigma_{r+1}}{\sigma_r}\cdot\frac{2\Bar\sigma_{r+1}}{\sigma_{r+1}}\sqrt{\epsilon}\nonumber\\
    &= c \cdot \frac{2\Bar\sigma_{r+1}}{\sigma_{r+1}}\sqrt{\epsilon}.\label{eqn:err_ineq}
\end{align}
The bound $\frac{\sigma_{r+1}}{\sigma_r}<\sqrt{1-\delta^2}$ implies that $c := \frac{\sigma_{r+1}}{\sqrt{1-\delta^2}\sigma_r} < 1$. \\
 (a) If
$\dist(A, U^{(r)}) > \frac{c}{1-c}\frac{2\Bar\sigma_{r+1}}{\sigma_{r+1}}\sqrt{\epsilon}$
or equivalently if
$\epsilon < \Big( \frac{\sigma_{r+1}}{2\Bar\sigma_{r+1}}\frac{(1-c)}{c}\dist(A, U^{(r)}) \Big)^2$,
then
\begin{align*}
    e &< c \cdot \frac{2\Bar\sigma_{r+1}}{\sigma_{r+1}}\sqrt{\epsilon}\\
    &< (1-c)\cdot \dist(A, U^{(r)}).
\end{align*}
Hence  (\ref{eqn:con_dist2}) implies that
\begin{align*}
\dist(B^T, V^{(r)}) &\le c\cdot \dist(A,U^{(r)}) + e\\
&< \dist(A,U^{(r)}).
\end{align*}

\medskip

\noindent (b) If
$\dist(A, U^{(r)}) \le \frac{c}{1-c}\frac{2\Bar\sigma_{r+1}}{\sigma_{r+1}}\sqrt{\epsilon}$,
then  (\ref{eqn:con_dist2}) implies that
\begin{align}
\dist(B^T, V^{(r)}) &\le c\cdot \dist(A,U^{(r)}) + e\nonumber\\
&< \frac{c^2}{1-c}\frac{2\Bar\sigma_{r+1}}{\sigma_{r+1}}\sqrt{\epsilon} + c \cdot \frac{2\Bar\sigma_{r+1}}{\sigma_{r+1}}\sqrt{\epsilon}\nonumber\\
&<\frac{1}{1/c-1}\frac{2\Bar\sigma_{r+1}}{\sigma_{r+1}}\sqrt{\epsilon}\\ \label{eqn:error_term}
&=\frac{\sigma_{r+1}}{\sigma_r\sqrt{1-\delta^2} - \sigma_{r+1}}\frac{2\Bar\sigma_{r+1}}{\sigma_{r+1}}\sqrt{\epsilon}\nonumber\\
&=\frac{2\Bar\sigma_{r+1}}{\sigma_r\sqrt{1-\delta^2} - \sigma_{r+1}}\sqrt{\epsilon}.\nonumber
\end{align}
\end{proof}

Next we strengthen the estimate of  Lemma \ref{lemma:alt_ref} under some additional assumptions on the ratio $\sigma_r/\sigma_{r+1}$ and 
$\delta$.

\begin{corollary}\label{coro:iter_ref}
Let $m, n, r, \epsilon, \delta$, $M$, $U^{(r)}$, $V^{(r)}$, $A$, $B$, $c$, and $e$ be defined as in Thm. \ref{thm:contracting_distance}, such that   (\ref{eqn:con_dist1}) and (\ref{eqn:con_dist2}) hold.
Let $\delta:=\dist(A,U^{(r)}) \le 0.5$, $\theta:=
\frac{\Bar\sigma_{r+1}}{\sigma_{r+1}}$, and $\frac{\sigma_{r+1}}{\sigma_r} \le 0.5$.
Then 
\begin{equation}\label{eqtht}
\dist(B^T, V^{(r)}) <\max \Big\{ 0.87\cdot\dist(A,U^{(r)})
,4\theta\sqrt{\epsilon}\Big\}~.
\end{equation}
\end{corollary}
\begin{proof}
Bounds $\delta \le 0.5$ and $\frac{\sigma_{r+1}}{\sigma_r} \le 0.5$ together
imply that $c := \frac{\sigma_{r+1}}{\sqrt{1-\delta^2}\sigma_r} \le \frac{1}{\sqrt{3}}$. Substitute this  inequality into  (\ref{eqn:err_ineq}) to obtain
\begin{equation}
e < \frac{2}{\sqrt{3}}\theta\sqrt{\epsilon}.
\end{equation}
Apply the same argument as in Lemma \ref{lemma:alt_ref} in the case where $\dist(A, U^{(r)}) \ge 4\theta\sqrt{\epsilon}$ 
and obtain
\begin{align*}
\dist(B^T, V^{(r)}) &\le c\cdot \dist(A,U^{(r)}) + e\\
&< \frac{1}{\sqrt{3}}\dist(A,U^{(r)}) + \frac{1}{4}\cdot\frac{2}{\sqrt{3}}\dist(A,U^{(r)})\\
&< 0.87\cdot\dist(A,U^{(r)}),
\end{align*}
Likewise, in the case where $\dist(A, U^{(r)}) < 4\theta\sqrt{\epsilon}$
 obtain
\begin{align*}
\dist(B^T, V^{(r)}) &\le c\cdot \dist(A,U^{(r)}) + e\\
&< \frac{4}{\sqrt{3}}\cdot\theta\sqrt{\epsilon}
+ \frac{2}{\sqrt{3}}\cdot\theta\sqrt{\epsilon}=2\theta\sqrt{3\epsilon}\\
&< 4\theta\sqrt{\epsilon},
\end{align*}
\end{proof}

\begin{remark} \label{regap}
Apply
Corollary \ref{coro:iter_ref}  
 to the pairs $A:=A_t$, $B:=B_{t+1}$ 
 and $A:=B_{t+1}$, $B:=A_{t+1}$
 for $t=0,1,\dots,\tau-1$  to  yield
linear decrease whp of all distances  appearing in the algorithm up to or beyond  the value $4\theta\sqrt {\epsilon}$.
This value   is quite large 
if the tail of the spectrum of the singular values of $M$ is flat, e.g.,  $4\theta\sqrt {\epsilon}\ge 1296r\beta^{-1}\xi^{-1}\frac{\sqrt{q-r}}{\sqrt q}$ for $q:=\min\{m,n\}$  if
$\sigma_{r+1}=\sigma_q$: indeed, in this case 
$\theta=\sqrt {q-r}$,
 while we cannot
decrease $\epsilon^2$ below
$1296\beta^{-1}r^
2\xi^{-4}/q$ because of bounds (\ref{eql}) and  $l\le q$. However, $\theta\approx 1$, while $4\sqrt
{\epsilon}$ has order of $r/\sqrt q$ if $\sigma_{r+2}$ is a small fraction of $\sigma_{r+1}$.
\end{remark}

\subsection{Main Theorem: an output error bound}
\begin{theorem}\label{thm:altermain}
Suppose that 
 $m, n, r$, $M$, $U^{(r)}$, $V^{(r)}$ are defined as in Thm. \ref{thm:contracting_distance},  
$$\frac{\sigma_{r+1}(M)}{\sigma_{r}(M)} \le \frac{1}{2},~  
 A_0\in\mathbb R^{m\times r},~{\rm and}~ 
 \dist(A_0, U^{(r)}) \le \frac{1}{2}.$$
%~{\rm so~that}~1\le \theta^2\le \min\{m,n\}-r,$$
Fix two (sufficiently small) positive numbers $\xi$
  and $\epsilon$ such that
$$\xi  < 1/3~{\rm and}~\epsilon \le (8\theta)^{-2}, $$
for $\theta:= \frac{\Bar\sigma_{r+1}}{\sigma_{r+1}}$ 
 of (\ref{eqtht}), and  let   
Alg. \ref{algalter}
be applied
for  $\tau= \ceil{\frac{1}{2}\log_{0.87}(8\theta\cdot\epsilon)} $.
Then its output  matrix
 $A_{\tau}$ satisfies the bound
 \begin{equation}
    \dist (A_{\tau}, U^{(r)}) \le  4\theta\cdot\sqrt{\epsilon}
 \end{equation}
 with a probability no less than $1 - 2\tau\cdot\xi$.
\end{theorem}

\begin{proof} 
First we prove by induction in $t$ that for  $t=1,2,\dots,\tau$ the matrices $A_t$
and  $B_t$ computed by
 Alg. \ref{algalter} satisfy whp the following bounds:  \begin{equation}
\dist(A_t, U^{(r)}) < 1/2
\textrm{ and } \dist(B_{t+1}^T, V^{(r)}) < 1/2\label{eqn:mainthmeq1}. 
\end{equation}
We immediately verify that 
$\dist(A_0, U^{(r)}) < 1/2$. 
Now let $\dist(A_t, U^{(r)}) < 1/2$
for a fixed $t$. Then  apply Corollary \ref{coro:iter_ref}
for $A=A_t$ and obtain that 
$$\dist(B_{t+1}^T, V^{(r)})\le \max\{0.75\dist(A_t, U^{(r)}),4\theta\sqrt{\epsilon}\}$$ with a probability no less than $1-\xi$. This implies the second inequality of (\ref{eqn:mainthmeq1})
because $\epsilon < (8\theta)^{-2}$  by assumption and hence $4\theta\sqrt{\epsilon} < 1/2$.

%VP1_30_25 Similarly deduce from Corollary \ref{coro:iter_ref} the second inequality of (\ref{eqn:mainthmeq1}) whp. 

Apply the union bound for $\tau = \ceil{\frac{1}{2}\log_{0.87}(\frac{4\theta\cdot\epsilon}{1/2})} = \ceil{\frac{1}{2}\log_{0.87}(8\theta\cdot\epsilon)}$ and
obtain that, with  
 a probability no less than $1 - 2\tau\cdot\xi$, 
all principal angle distances of (\ref{eqn:mainthmeq1})
 for all $t$  
 either decrease by at least a factor of 1/0.87 versus the previous ones
or are less  than $4\theta\cdot\sqrt{\epsilon}$.
\end{proof}

%------------------------------------------------------------------------------

\section{Numerical Tests}\label{snmrtsts}

\subsection{Overview}\label{sovrv}  

We implemented our algorithms  in Python with Numpy and Scipy packages. For solving generalized LLSP we call {\bf lstsq} and for computing Rank Revealing QR factorization we call {\bf qr}, relying on {\bf Lapack} function {\bf gelsd} and {\bf dgeqp3}, respectively.

We  performed all tests
on a machine running Mac OS 10.15.7 with 2.6 GHz Intel Core i7 CPU and 16GB of Memory.

\subsection{Input matrices}\label{ststmt} 

We used the following classes of input matrices
(GitHub repo https://github.com/soohgo/superfast-norm-est):

\begin{itemize}
\item %VP??We generated 
A $3000\times 3000$  {\it fast-decay} matrix  %VP?as the product 
$U\Sigma V^T$ where  $U$ and $V$  are the matrices of the left and right singular vectors of the SVD of a
$3000\times 3000$  Gaussian random 
matrix,\footnote{filled with independent   standard  Gaussian (normal) random variables} respectively, and  where
$\Sigma = {\rm Diag}(\sigma_i)_{i=1}^{3000}$ for $\sigma_i = 1$ for $i \le r$ and $\sigma_i = 2^{-(i-r)}$ for $i > r$. 
 %VP??We set target rank $r$ to 10.

\item  %VP??We generated 
A $3000\times 3000$ {\it slow-decay}  matrix;
 %VP?? as a similar 
it denotes the same product $U\Sigma V^T$ except that 
 now we let
%VP?? as the product $U\Sigma V^T$ for $U$ and $V$ generated  above and for $\Sigma = {\rm Diag}(\sigma_i)_{i=1}^{3000}$ with
  $\sigma_i = 1$ for $i \le r$ and $\sigma_i = (1+ i-r)^{-2}$ for $i > r$.
   %VP? ?We set  target rank $r$ to 10.

\item  %VP??We generated 
A $3000\times 3000$  {\it single-layer potential}  matrix;   %VP??by
it discretizes a single-layer potential operator of  \cite[Sec. 7.1]{HMT11}
$$ 
[Sf](x) = \int_{\Gamma_1} \log |x- y| \cdot f ~\textrm{d}\sigma(y), ~x\in \Gamma_2,
$$
for a constant  $f$, for two curves  $\Gamma_1(t) =\big(r(t)\cos (t), r(t)\sin (t) \big)$ and $\Gamma_2 = (3\cos (t), 3\sin (t))$ in $\mathbb{R}^2$, for $t \in [0, 2\pi]$ and $r(t) = 2.5 + \cos (3t)$.
 %VP??We set  target rank $r$ to 11.

\item   %VP??We generated 
A $2000\times 2000$ {\it Cauchy} matrix $(\frac{1}{X_i-Y_j})_{i,j=1}^{2000}$.  Here $X_i$ and $Y_j$ denote independent random variables $X_i$ and $Y_j$ uniformly distributed on the intervals $(0, 100)$ and $(100, 200)$, respectively. 
This matrix has  fast decaying singular values (cf. \cite{BT17}). 
 %VP??We set  target rank $r$ to 10.

\item  %VP??We defined 
A $1000\times 1000$ {\it shaw} matrix. It 
 %VP??by means of  
discretizes a one-dimensional image restoration model.\footnote{See 
%database at 
 http://www.math.sjsu.edu/singular/matrices and 
  http://www2.imm.dtu.dk/$\sim$pch/Regutools 
  
For more details see Chapter 4 of the Regularization Tools Manual at \\
  http://www.imm.dtu.dk/$\sim$pcha/Regutools/RTv4manual.pdf} 
  
 %VP?We set target rank $r$ to 10.
 We set target rank $r$ to 11
 in the case of the single layer potential matrix and set the rank to 10 for all other input matrices.
\end{itemize}   

\subsection{General description of numerical experiments}\label{ststsitreflsc}

We computed an initial crude LRA (rank-$r$ approximation) of an input matrix $M$ in  two ways: 
%QL0207
(i) First apply the  Range Finder (\cite[Alg. 4.1]{HMT11}) with no oversampling
to compute an orthogonal matrix $Q\in \mathbb{R}^{m\times r}$,
given an input matrix  of the size $m\times n$ and a target rank  $r$; 
then output rank-$r$ LRA $QQ^TM$ (this initialization is not superfast). 

(ii) First apply a  single  ACA loop  
to compute CUR LRA  $CG^{-1}R$, where the generator matrix $G$ is an $r\times r$ submatrix of $M$  and 
%VP??where 
$C$ and $R$ are corresponding column and row submatrices of $M$, respectively. 
Namely,  
 first compute %%VP
 Strong RRQR 
 (Rank Revealing QR) factorization of randomly selected column submatrix $C_0$; the resulting permutation matrix determines a row index set and a  submatrix $R$. Then compute Strong 
 %%V
 RRQR factorization of the matrix $R$; the resulting permutation matrix determines a column index set and hence  matrix $C$ and generator $G$.
 The computation of
 %%V 
Strong RRQR factorization of skinny matrices is superfast  for $r^2\ll \min\{m,n\}$.
Continue 
 by performing
 %just
  a small number of ACA steps, 
 leaving some room for refinement and keeping the computation superfast,  so that our refinement cost is comparable with the cost of ACA initialization, with which we have still consistently    improved an initial LRA significantly 
 in two or three refinement steps.

We applied Alg. \ref{algalter} separately to  the two initial approximations (i) and (ii) above. For each alternating refinement step we sampled $d$ columns or $d$ rows, for  $d=15r$, which turned out to be sufficiently large  in our tests, even though 
 for the formal support of our refinement results we need  much larger samples of order $r^2\epsilon^{-2}$.  

We computed Frobenius relative error norm $\frac{|| M - A_tB_t||_F}{||M - M_r||_F}$ for the initial LRA $A_0B_0$ and
 the  LRAs $A_tB_t$ computed at the $t$th refinement steps for $t=1,2,3,4,5$.
 
We repeated the experiments 50 times and 
displayed the average error ratios in Table \ref{tableitref}.
  
\begin{table}[ht]

\begin{center}
\begin{tabular}{c||c|c c c c c }
 & \multicolumn{6}{c}{ HMT11 + Refinement }\\
\hline
Input & Init. & 1st step & 2nd step & 3rd step & 4th step & 5th step\\
\hline
shaw & 9.2486 & 1.3920 & 1.1726 & 1.0892 & 1.0727 & 1.0772 \\ 
SLP   & 3.5421 & 1.4720 & 1.1462 & 1.0971 & 1.0912 & 1.0825 \\ 
Cauchy & 5.7180 & 1.4783 & 1.1383 & 1.0764 & 1.0826 & 1.0747 \\ 
slow decay & 3.1190 & 1.7194 & 1.0826 & 1.0726 & 1.0715 & 1.0680 \\ 
fast decay  & 2.0612 & 1.6596 & 1.2429 & 1.1054 & 1.0756 & 1.0735 \\ 
\hline
\hline
 & \multicolumn{6}{c}{ Cross-Approximation + Refinement } \\
\hline
Input & Init. & 1st step & 2nd step & 3rd step & 4th step & 5th step\\
\hline
 shaw & 8.5939 & 1.0782 & 1.0723 & 1.0754 & 1.0674 & 1.0752 \\ 
SLP    & 6.2216 & 1.4372 & 1.1017 & 1.0796 & 1.0764 & 1.0782 \\ 
Cauchy & 1703.7035 & 2.4336 & 1.1003 & 1.0829 & 1.0797 & 1.0773 \\ 
slow decay & 2.5236 & 1.1128 & 1.0691 & 1.0726 & 1.0710 & 1.0683 \\ 
fast decay & 2.0717 & 1.3539 & 1.1595 & 1.0898 & 1.0743 & 1.0752 
\end{tabular}
\end{center}
\caption{Reducing Error Ratio in Refinements Alg. \ref{algalter}}\label{tableitref}
\end{table}

For both choices of initial approximation, the tests show  that
already 
%%vic in its
in a few steps Alg. \ref{algalter} have significantly decreased the initial relative error norm,  which has
%vic
been more or less stabilized in the 
 subsequent steps. 
 
\subsection{Three options for solving generalized LLSP}\label{srfLRA w3LLSP} 

In these tests  we started with uniformly and randomly selected $r$ columns of $M$  
and then performed iterative refinement of LRA by using three different LLSP algorithms: (a) superfast  sampling directed by leverage scores (Alg. \ref{algalter}), (b) subspace embedding with a Gaussian multiplier, and (c) the standard Linear Least Squares Solver.

(b) The subspace embedding with Gaussian multiplier is a well-known %approximation 
 approach, where in each iteration 
one computes $B_t = \argmin_X ||GA_tX - GM ||_F$ and  $A_{t+1} = \argmin_X ||XB_tH - MH||_F$
and where $G$ and $H$ are $d\times m$ and $n\times d$ Gaussian matrices, respectively, generated independently of one another. This computation is not superfast.
 
(c) The standard deterministic Linear Least Squares Solver computes 
$$ \textstyle B_t = \argmin_X ||A_tX - M||_F~{\rm and}~A_{t+1} = \argmin_X ||XB_t - M||_F;$$
this  algorithm is more costly but outputs an optimal solution. 

\subsection{Test results}
 
In each iteration we computed the average principal angle distance (between the 
%VP?
matrices of the 
$r$-top   left singular vectors of the input matrix and its approximation) and the average run time from 10 repeated tests;  then we plotted  our tests results in Fig. \ref{fig:iter_ref_lev_scr}. 
%VP I had to delete the Figure becaus it was not compiled
The tests showed that, similarly to the relative error norm in the tests of the previous subsections, the principal angle distance also
decreased significantly 
in a  few iterations  and 
then stabilized. 

The principal angle distance would have converged to 0 if the computations were exact and if we solved all auxiliary generalized LLSPs in (\ref{eqllsps})  by applying the standard deterministic algorithm (case (c)), but  the distance was
still expected to converge to a small positive 
%VP? constant
value if the approximation algorithms using Gaussian embedding (case (b)) or sampling directed  by leverage scores (case (a)) were applied for generalized  LLSPs.
 This small positive value depends on the embedding dimension and the number of samples.

The tests have confirmed such expected behavior even in the presence of rounding errors: the refinement based on the standard LLSP solver decreased  the principal angle distance considerably stronger than with the two other approaches, (a) and (b), which has output   CUR LRAs (unlike the case (c) of using the Standard Algorithm), and they were still  close to optimal.

%VP? 
%VP? The next comments do not seem to be supported by test results and  Itentatively deleted them:
%For the {\it Single Layer Potentials} input, the Principal Angle Distance of the refinement with Gaussian Embedding and the refinement with Leverage Scores are slightly worse than the approxi\-mations of other %inputs. This could   
% be caused by the ``heavier'' tail singular values. 
%However, the approximation remains rather close to the optimal in terms of Error Ratio regardless. 

According to our tests for running time,
in cases (a) and (b)  the algorithms  have  run significantly faster than  in case (c). 
%%vic
 Moreover, the algorithm with leverage scores (case (a)) used considerably less time at each iteration 
than  the Gaussian Embedding approach (b), and this benefit  strengthened  
 as the size of the input matrix grew. 
 %VP? I am not sure here because in our tests the matrix sizes were fixed. Shoud we remove the claim? 
 %VP07 I leave this to you: remove or not.
% Even more important,  in our tests the improvement came at almost no cost in terms of the growth of computational  precision.
%%%VP?? Should we expand the latter claim or specify it?
%%%Qi I think we should removed this. We mostly focuse on the theoretical aspect,
 % and in my understanding switching between low/high precision aims at
 % refining the last few digits due to in exactness of floating point arithematic

%QL-LAA-2
% I re-drew the test result with log-scale following the comment of one of the reviewer.
% Originally the three curves of the 3 graphs on first row 
% are close straight lines because their values are both close to 0.
% With log-scale, the graphs show standard LLSP solver has advantage over
% the approximation methods in terms of precision. 
% The advantage of the leverage scores sampling method has advantage of efficiency,
% which is shown in the second row. 

%VP I had to delete the Figure becaus it was not compiled! Could you help?
%QL07 figure added back  
\begin{figure}[ht]
  \centering
  \includegraphics[width = \textwidth]{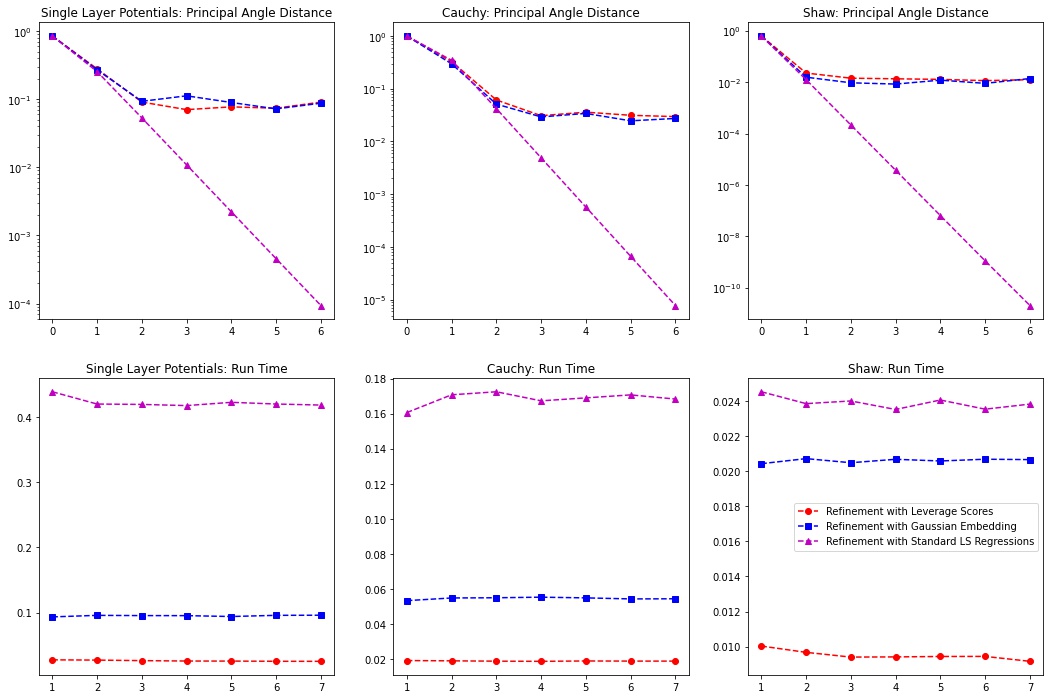}
\caption{%V? edit  begins Reducing
%V?  edit  ends
  Principal Angle Distance with Least Squares Regression in Refinement  Alg. \ref{algalter}}
   \label{fig:iter_ref_lev_scr}
\end{figure}

%------------------------------------------------------------------------------

\section{Conclusions}\label{sconc}

%----------------------------------------------------------------------------
%\medskip

We devised and analyzed two ALS
 algorithms for 
iterative refinement of an LRA. One of them is deterministic, another is randomized and superfast.
We proved that
for a large specified  class of input matrices the deterministic algorithm  outputs meaningful or even near-optimal LRA and that the randomized  superfast algorithm  outputs near-optimal CUR LRA whp for a little narrower matrix class. The latter algorithm recursively invokes the algorithm of \cite{DMM08}, which formally requires quite large samples of rows and columns of an input matrix, and this requirement is translated to our case; in the tests on real world matrices, however, both randomized  algorithms of \cite{DMM08}
 and ours  applied with much smaller samples have consistently remained efficient.

In our tests of the randomized algorithm on real world matrices the  error norm of initial LRAs 
tended to decrease significantly  already in two or three iterations and then stayed more or less stable.
 Formal support for cited
empirical behavior of our algorithm is a challenge.

Our initial  assumptions used for proving convergence  are quite restrictive even for our deterministic algorithm. Can we relax these
restrictions?
Presently we are
working  towards this goal by means of incorporation of some known techniques for LRA.

 We hope that already our current paper  motivates further study of superfast LRA and reveals
some power of LRA hidden in  \cite{DMM08}. 

%For two related challenges, we point out formal support of empirical power of superfast combination of ACA with the algorithms of \cite{DMM08}, observed in \cite{PLSZa}, and extension of classical iterative refinement techniques to LRA, initially studied in \cite{PLa} and currently in \cite{GKLPPZa}.
      
\bigskip
%------------------------------------------------------------------------------
%\clearpage

{\bf \Large Appendix} 
\appendix

%------------------------------------------------------------------------------
%V??\section{Background on randomized matrix computations}\label{ssdef}

%------------------------------------------------------------------------------ 
                                                     
%V??{\bf A.1. Nondegeneration of a Gaussian Matrix.}
%------------------------------------------------------------------------------
%------------------------------------------------------------------------------

\section{Proof of Thm. \ref{thm:dmm08thm5}}\label{appendix:pfdmm08thm5}
Recall that the matrix $A\in\mathbb{R}^{m\times r}$ has rank $r$
and let $A = U\Sigma V^T$ be SVD.

Fix $U_\perp\in\mathbb{R}^{m\times(m-r)}$ such that matrix $\begin{bmatrix}
    U & U_{\perp}
\end{bmatrix}$ is orthogonal.

Let Algorithm \ref{algsmplex}  
 compute a sampling  and scaling matrix $S$ by  using row leverage scores of $A$ and $l$ samples.
Notice that $\mathcal{S}$ is also associated with the sampling distribution of row leverage scores of the matrix $U$.

\begin{lemma}\label{lea1}
% Then we show that with high probability .
Fix positive  $\epsilon, \beta$,  and $\xi,$ and a positive integer $l$ such that $\epsilon,    \beta, \xi < 1$ and  $l \ge \frac{4r^2}{\epsilon^2\beta\xi^4}$. Then with a probability no less than $1-\xi$ it holds that
{\rm (1)} $\mathcal{S}A$ has full rank $r$ and
{\rm (2)} $||(\mathcal{S}U)^+ - (\mathcal{S}U)^T||_2 \le \epsilon\xi/\sqrt{2}$.
\end{lemma}
\begin{proof}
Proceed as in the proof of \cite[Lemma 1]{DMM08} but apply Markov's inequality such that with a probability no less than $1-\xi$ it holds that
\begin{align}
    ||U^TU - U^T\mathcal{S}^T\mathcal{S}U||_F &\le \frac{1}{\xi} \mathop{{}\mathbb{E}}\bigg[||U^TU - U^T\mathcal{S}^T\mathcal{S}U||_F\bigg] \nonumber\\ 
    &\le \frac{1}{\xi\sqrt{\beta l}} ||U||^2_F \nonumber\\
    &= \frac{r}{\xi\sqrt{\beta l}} \nonumber\\ 
    &\le \frac{\epsilon\xi}{2}  \label{eqn:lessepsxi}\\
    &< \frac{\epsilon}{2}~. \nonumber
\end{align}
Obtain from this inequality that $|1 - \sigma_j^2(\mathcal{S}U)|\le ||U^TU - U^T\mathcal{S}^T\mathcal{S}U||_2 \le ||U^TU - U^T\mathcal{S}^T\mathcal{S}U||_F < \epsilon/2$.  Since $\epsilon<1$  by assumption, it follows that $\sigma_j(\mathcal{S}U) > 1/\sqrt{2}$ for all $j\le r$ and hence $\rank(\mathcal{S}U) = r$.

 Now apply  \cite[Lemma 1]{DMM08} and deduces that
\begin{align*}    
||(\mathcal{S}U)^+ - (\mathcal{S}U)^T||_2 &= ||\Sigma_{\mathcal{S}U}^{-1} - \Sigma_{\mathcal{S}U}||_2\\
&= \max_j \Big|\sigma_j(\mathcal{S}U) - \frac{1}{\sigma_j(\mathcal{S}U)}\Big|\\
&= \max_j \Big|\frac{1-\sigma_j^2(\mathcal{S}U)}{\sigma_j(\mathcal{S}U)}\Big|\\
&\le \frac{\epsilon\xi/2}{1/\sqrt{2}} = \epsilon\xi/\sqrt{2}.
\end{align*}
The latter inequality follows from  (\ref{eqn:lessepsxi}) and the bound $\sigma_j(\mathcal{S}U) \ge 1/\sqrt{2}$.
\end{proof}

\begin{lemma}
For $\epsilon, \beta, \xi,$ and $l$ as in Lemma 
\ref{lea1} it holds with a probability no less than $1-\xi$ that
$$
||U^T\mathcal{S}^T\mathcal{S}U_{\perp}U_{\perp}^TM||_F \le \frac{\epsilon\xi}{2\sqrt{r}}||U_{\perp}U_{\perp}^TM||_F.
$$
\end{lemma}
\begin{proof}
Since $U$ is orthogonal, obtain that
\begin{align*}
    ||U^T\mathcal{S}^T\mathcal{S}U_{\perp}U_{\perp}^TM||_F &= ||UU^T\mathcal{S}^T\mathcal{S}U_{\perp}U_{\perp}^TM||_F \\
    &= ||UU^TU_{\perp}U_{\perp}^TM - UU^T\mathcal{S}^T\mathcal{S}U_{\perp}U_{\perp}^TM||_F.    
\end{align*}
The latter equation holds  because $U^TU_{\perp} = \mathbf{0}$.
 
Now substitute $UU^T=A$, $U_{\perp}U^T_{\perp}M=B$,
$UU^T\mathcal{S}^T=C$, and $\mathcal{S}U_{\perp}U^T_{\perp}M=R$.

Recall that the probability distribution
$\tilde p_1,\dots, \tilde p_m$   used for the construction of the sampling and scaling matrix $S$ in Alg. \ref{algsmplex}
has been computed from row leverage scores of the matrix  $A=A_t$.
Hence this  distribution satisfies (\ref{eqlevsc_row}),
and we can apply [9, Thm. 6] and obtain
\begin{align*}
\mathop{{}\mathbb{E}}||UU^TU_{\perp}U_{\perp}^TM - UU^T\mathcal{S}^T\mathcal{S}U_{\perp}U_{\perp}^TM||_F &= \mathop{{}\mathbb{E}}||AB - CR||_F \\
&\le \frac{1}{\sqrt{\beta l}}||A||_F||B||_F\\
&=\frac{\sqrt{r}}{\sqrt{\beta l}}||U_{\perp}U_{\perp}^TM||_F.
\end{align*}

 Markov's inequality  implies that 
\begin{align*}
||UU^TU_{\perp}U_{\perp}^TM - UU^T\mathcal{S}^T\mathcal{S}U_{\perp}U_{\perp}^TM||_F &\le \frac{1}{\xi}\mathop{{}\mathbb{E}}||UU^TU_{\perp}U_{\perp}^TM - UU^T\mathcal{S}^T\mathcal{S}U_{\perp}U_{\perp}^TM||_F\\
&\le \frac{\sqrt{r}}{\xi\sqrt{\beta l}}||U_{\perp}U_{\perp}^TM||_F
\end{align*}
with a probability no less than $1 - \xi$.
 
 To complete the proof recall that $\xi<1$ and replace $l$ with its upper bound. 
\end{proof}

\noindent Now we are ready to {\bf prove Thm. \ref{thm:dmm08thm5}}
by following the  argument  in \cite[Proof of Eqn. (21)]{DMM08}.

With a probability at least $1-2\xi$,
the assumptions of  the previous  two lemmas hold, and so
\begin{align*}
    M - A\tilde{X} &= M - A(\mathcal{S}A)^+\mathcal{S}M \\
    &= M - U(\mathcal{S}U)^+\mathcal{S}M \hspace{1cm} \textrm{(since $\mathcal{S}A$ has full rank)}\\
    &= M - U(\mathcal{S}U)^+\mathcal{S}(UU^T + U_{\perp}U_{\perp}^T)M \\
    &= M - U(\mathcal{S}U)^+\mathcal{S}UU^TM - U(\mathcal{S}U)^+\mathcal{S}U_{\perp}U_{\perp}^TM \\
    &= M - UU^TM - U(\mathcal{S}U)^+\mathcal{S}U_{\perp}U_{\perp}^TM \\
    &= U_{\perp}U_{\perp}^TM - U(\mathcal{S}U)^+\mathcal{S}U_{\perp}U_{\perp}^TM \\
    &= U_{\perp}U_{\perp}^TM - U(\mathcal{S}U)^T\mathcal{S}U_{\perp}U_{\perp}^TM -U((\mathcal{S}U)^T-(\mathcal{S}U)^+)\mathcal{S}U_{\perp}U_{\perp}^TM. 
\end{align*}
Take the Frobenius norm of both sides,  substitute
$\Omega = (\mathcal{S}U)^+ - (\mathcal{S}U)^T$, and obtain
\begin{align*}
    ||M - A\tilde{X}||_F &\le ||U_{\perp}U_{\perp}^TM||_F + ||U(\mathcal{S}U)^T\mathcal{S}U_{\perp}U_{\perp}^TM||_F + ||U\Omega\mathcal{S}U_{\perp}U_{\perp}^TM||_F\\
    &\le ||U_{\perp}U_{\perp}^TM||_F + ||U^T\mathcal{S}^T\mathcal{S}U_{\perp}U_{\perp}^TM||_F + ||\Omega||_2\cdot||\mathcal{S}U_{\perp}U_{\perp}^TM||_F \label{pf41inequal}
\end{align*}

\cite[Lemma 3]{DMM08} shows that
$\mathop{{}\mathbb{E}}
||\mathcal{S}U_{\perp}U_{\perp}^TM||_F^2 = ||U_{\perp}U_{\perp}^TM||_F^2$.

Combine Jensen's and Markov's inequalities 
and obtain that  
$$
||\mathcal{S}U_{\perp}U_{\perp}^TM||_F \le \frac{1}{\xi} ||U_{\perp}U_{\perp}^TM||_F
$$
with a probability at least $1-\xi$.
 
Now  combine the  latter pair of inequalities and obtain that  

\begin{align*}
    ||M - A\tilde{X}||_F &\le ||U_{\perp}U_{\perp}^TM||_F + ||U^T\mathcal{S}^T\mathcal{S}U_{\perp}U_{\perp}^TM||_F + ||\Omega||_2\cdot||\mathcal{S}U_{\perp}U_{\perp}^TM||_F \\
    &\le ||U_{\perp}U_{\perp}^TM||_F + \frac{\epsilon\xi}{2\sqrt{r}}||U_{\perp}U_{\perp}^TM||_F + \frac{\epsilon\xi}{\sqrt{2}} \cdot \frac{1}{\xi} ||U_{\perp}U_{\perp}^TM||_F\\
    &\le (1 + \frac{\epsilon\xi}{2\sqrt{r}} + \frac{\epsilon}{\sqrt{2}})||U_{\perp}U_{\perp}^TM||_F.
\end{align*}
with a probability no less than $1-3\xi$. 

We proved this bound for any pair of positive 
$\epsilon$ and 
$\xi$ exceeded by 1, and so it still holds if we
 replace $\xi$ with $\xi/3$ and $\epsilon$ with $\epsilon/2$. Now, since $\frac{\epsilon/2\cdot\xi/3}{2\sqrt{r}} + \frac{\epsilon/2}{\sqrt{2}} \le \frac{\epsilon}{16} + \frac{\epsilon}{2\sqrt{2}} \le \epsilon$, conclude that for $\epsilon, \beta, \xi$ and $l$ such that $\epsilon < \frac{1}{2}$, $\beta < 1$, $\xi < \frac{1}{3}$, and $l \ge \frac{4r^2}{\epsilon^2\beta\xi^4} \cdot 2^2 \cdot 3^4 = 1296r^2\beta^{-1}\epsilon^{-2}\xi^{-4}$,
 the bound
$$
||M - A\tilde{X}||_F \le (1+\epsilon)||U_{\perp}U_{\perp}^TM||_F = (1+\epsilon)||M - AA^+M||_F
$$
holds with a probability no less than $1-\xi$.                                                                                                                                                                                                                                                                                                                                                                                                                                                          
 This completes the proof of Thm. \ref{thm:dmm08thm5}.
%\medskip

%\medskip
%\clearpage

%----------------------------------------------------------------------------
                                                                                                                                                                                                                                                                                                                                                                                                                                                           
\medskip

%-----------------------------------------------------------------------------

\noindent {\bf Acknowledgements:}
Our work has been supported by NSF Grants 
 CCF--1563942 and CCF--1733834
and PSC CUNY Award  66720-00 54.
Years ago, E. E. Tyrtyshnikov,
citing his discussion with M. M. Mahoney, 
challenged the second author
 to combine ACA
 iterations with randomized LRA algorithms. Taking this challenge eventually evolved into our current ALS extension of the seminal work \cite{DMM08}.
%We are also grateful to
%%vic
%to N. L. Zamarashkin for his comments on his work with A. Osinsky  on LRA via volume maximization and on the first draft of  \cite{PLSZ17}, and to
%  S. A. Goreinov, I. V. Oseledets, A. Osinsky,  E. E. Tyrtyshnikov, and N. L. Zamarashkin for  reprints and pointers to relevant bibliography and 
 Thoughtful  comments
of reviewers
helped us greatly improve our original draft.

%------------------------------------------------------------------------------

%------------------------------------------------------------------------------

%1------------------------------------------------------------------------------

\end{document}